\documentclass[reqno,a4paper]{amsart}
\usepackage{eucal,amsfonts,amssymb,amsmath,amsthm,epsfig,mathrsfs}
\usepackage{color}

\usepackage{amscd,amsxtra}
\usepackage{enumerate}
\usepackage{latexsym}

\allowdisplaybreaks

 \makeatletter \@addtoreset{equation}{section}

\makeatother \makeatletter

\newtheorem{thm}{Theorem}[section]
\newtheorem{hyp}[thm]{Hypotheses}{\rm}
\newtheorem{hyp0}[thm]{Hypothesis}{\rm}
\newtheorem{lemm}[thm]{Lemma}

\newtheorem{prop}[thm]{Proposition}
\newtheorem{defi}[thm]{Definition}
\newtheorem{rmk}[thm]{Remark}{\rm}
\newtheorem{example}[thm]{Example}

\newcommand{\C}{{\mathbb C}}
\newcommand{\R}{{\mathbb R}}
\newcommand{\N}{{\mathbb N}}

\newcommand{\Rd}{\mathbb R^d}
\newcommand{\Om}{{\Omega}}

\newcommand{\A}{\mathcal{A}}

\newcommand{\B}{\mathcal{B}}

\newcommand{\bd}{\begin{defi}}
\newcommand{\ed}{\end{defi}}
\newcommand{\nnm}{\nonumber}
\newcommand{\be}{\begin{equation}}
\newcommand{\ee}{\end{equation}}
\newcommand{\barr}{\begin{array}}
\newcommand{\earr}{\end{array}}
\newcommand{\bmn}{\begin{eqnarray}}
\newcommand{\emn}{\end{eqnarray}}
\newcommand{\bnm}{\begin{eqnarray*}}
\newcommand{\enm}{\end{eqnarray*}}
\newcommand{\bln}{\begin{subequations}}
\newcommand{\eln}{\end{subequations}}

\newcommand{\ba}{\begin{align}}
\newcommand{\ea}{\end{align}}
\newcommand{\banm}{\begin{align*}}
\newcommand{\eanm}{\end{align*}}
\newcommand{\one}{\mbox{$1\!\!\!\;\mathrm{l}$}}

\title[Non autonomous parabolic problems]{Non autonomous parabolic problems with unbounded coefficients in unbounded domains}
\thanks{The authors are members of GNAMPA of the italian Istituto Nazionale di Alta Matematica. This work has been supported by
the M.I.U.R. Research Project PRIN 2010-2011 ``Problemi differenziali di evoluzione: approcci deterministici e stocastici e loro interazioni'' and
INdAM-GNAMPA Project 2014 ``Equazioni ellittiche e paraboliche a coefficienti illimitati''.}
\author{L. Angiuli}
\author{L. Lorenzi}
\address{L.A.: Dipartimento di Matematica e Fisica ``Ennio De Giorgi'', Universit\`a del Salento, Via Per Arnesano, I-73100 Lecce, Italy.}
\email{luciana.angiuli@unisalento.it}
\address{L.L.: Dipartimento di Matematica e Informatica, Universit\`a degli Studi di Parma, Parco Area delle Scienze 53/A, I-43124 Parma, Italy.}
\email{luca.lorenzi@unipr.it}

\begin{document}

\begin{abstract}
Given a class of nonautonomous elliptic operators $\A(t)$ with unbounded coefficients, defined in $\overline{I \times \Om}$
(where $I$ is a right-halfline or $I=\R$ and $\Om\subset \Rd$ is possibly unbounded), we
prove existence and uniqueness of the evolution operator
associated to $\A(t)$ in the space of bounded and continuous functions, under Dirichlet and first order, non tangential homogeneous boundary conditions. Some qualitative properties of the solutions, the compactness of the evolution operator and some uniform gradient estimates are then proved.
\end{abstract}

\maketitle

\section{Introduction}
Parabolic Cauchy problems with unbounded coefficients set in unbounded domains, with sufficiently smooth boundary, have been studied
in the autonomous case both in the case of homogeneous Dirichlet \cite{ForMetPri04Gra} and Neumann \cite{BerFor04Gra,BerForLor07Gra}
boundary conditions. On the other hand, the nonautonomous counterpart
have been studied, to the best of our knowledge, only in the particular case  $\Omega=\Rd_+$, again only
under homogeneous Dirichlet and Neumann boundary conditions \cite{AngLor13OnT}.

This paper is devoted to continue the analysis started in \cite{AngLor13OnT}, studying parabolic nonautonomous boundary Cauchy problems
with unbounded coefficients in a greater generality, with respect to both the domain, where the Cauchy problems are set, and the boundary
conditions considered. More precisely, let
$\Om \subset \Rd$ be an unbounded open set with a boundary of class $C^{2+\alpha}$, for some $\alpha\in (0,1)$, and let $I\subset \R$ be an open right halfline (possibly $I=\R$). For any fixed $s \in I$ and any $f\in C_b(\Omega)$ (the space of bounded and continuous functions on $\Omega$),
we consider the nonautonomous Cauchy problem
\begin{equation}
\label{NAPC}
\left\{
\begin{array}{ll}
D_tu(t,x)=(\A u)(t,x), &t\in(s,+\infty),\,x\in \Om,\\[1mm]
({\mathcal B}u)(t,x)=0,&t\in(s,+\infty),\,x\in \partial\Om,\\[1mm]
u(s,x)= f(x),& x\in \Om.
\end{array}
\right.
\tag{$P_{\mathcal B}$}
\end{equation}
The families of nondegenerate elliptic operators $\{\A(t)\}_{t \in I}$  and of boundary operators
$\{{\mathcal B}(t)\}_{t \in I}$ act on smooth functions $\zeta$ as follows:
\begin{equation}
(\A(t)\zeta)(x)=\sum_{i,j=1}^d q_{ij}(t,x)D_{ij}\zeta(x)+
\sum_{i=1}^d b_i(t,x)D_{i}\zeta(x)-c(t,x)\zeta(x),
\label{operator}
\end{equation}
for any $(t,x)\in I\times\Omega$, and
\begin{equation}
({\mathcal B}(t)\zeta)(x)=\sum_{i=1}^d\beta_i(t,x)D_i\zeta(x)+\gamma(t,x)\zeta(x),\qquad\;\,(t,x)\in I\times\partial\Omega.
\label{OpB}
\end{equation}
The coefficients of the previous operators are smooth enough functions, and all of them
but $\beta=(\beta_1,\ldots,\beta_d)$ may be unbounded; function $\beta$ either everywhere differs from $0$ on $\partial\Omega$ or therein identically vanishes.
In the first case, we assume the usual non-tangential condition, in the latter one, we assume that
$\gamma\equiv 1$ so that ${\mathcal B}\zeta$ is the trace of $\zeta$ on $\partial\Omega$.

We first prove existence and uniqueness of a bounded classical solution of problem \eqref{NAPC} (see Definition \ref{defi}).
The case $\gamma\ge 0$ requires rather weak assumptions on the coefficients of
the operators ${\mathcal A}(t)$ and ${\mathcal B}(t)$. No growth assumptions are assumed on the diffusion and drift coefficients
of the operators $\A(t)$, whereas the potential is assumed to be bounded from below, this condition being
not surprising at all since, as the autonomous case reveals: without any lower bound on the potential no bounded solutions
to problem \eqref{NAPC} exist in general.
Further, the existence of a so-called Lyapunov function $\varphi$, associated with the pair $(\mathcal{A}(t), \mathcal{B}(t))$
(cf. Hypothesis \ref{hyp-3}) is assumed, which serves as a fundamental tool to prove a maximum principle, which yields
uniqueness of the solution to problem \eqref{NAPC}. When $\gamma$ takes also negative values we assume an extra condition,
which is stated in terms of another Lyapunov function.
The existence and the uniqueness of a classical solution to problem \eqref{NAPC} allow us to define an evolution operator
$G_{\mathcal{B}}(t,s)$ of bounded linear operators in $C_b(\Om)$ and to prove some remarkable continuity properties that this evolution operator enjoys. As a consequence of the Riesz representation theorem and the continuity property of the evolution operator, we can show
that, for any $(t,s)\in\Lambda:=\{(t,s)\in I\times I:\,t>s\}$ and any $x\in\Omega$, there exists a finite Borel measure $g_{\mathcal B}(t,s,x,dy)$ such that
\begin{equation}\label{repr_intro}
(G_{\mathcal B}(t,s)f)(x)=\int_{\Om}f(y)g_{\mathcal B}(t,s,x,dy),\qquad\;\,f\in C_b(\Omega).
\end{equation}
Under an additional smoothness assumption on the diffusion coefficients we prove that
$G_{\mathcal B}(t,s)f$ admits an integral representation by means of a Green function $g_{\mathcal B}:\Lambda\times \Om\times \Om\to (0,+\infty)$, i.e., $g_{\mathcal B}(t,s,x,dy)=g_{\mathcal B}(t,s,x,y)dy$ for any $(t,s,x,y)\in\Lambda\times\Omega\times\Omega$.
For any fixed $s\in I$ and almost any $y \in \Om$, the function $g_{\mathcal B}(\cdot,s,\cdot,y)$ is smooth, satisfies $D_t g_{\mathcal B}-\A(t)g_{\mathcal B}=0$ in $(s,+\infty)\times \Om$.

Formula \eqref{repr_intro} plays a crucial role in the study of the compactness of the operator $G_{\mathcal B}(t,s)$ in $C_b(\Om)$. Indeed,
as the proof of Theorem \ref{thm-comp} reveals, the compactness of the operators $G_{\mathcal B}(t,s)$ in $C_b(\Om)$, for $(t,s)\in\Lambda\times J^2$, $J$ being a bounded interval,
follows from the tightness of the family of measures $\{g_{\mathcal B}(t,s,x,dy), \, x \in \Om\}$ for any $(t,s)\in\Lambda\cap J^2$. In view of this fact, a sufficient condition is then provided to guarantee the tightness of
the previous family of measures. Our result extends the results obtained in \cite{AngLor10Com,Lun10Com} in the case when $\Om=\Rd$.

Next, when the boundary operator ${\mathcal B}$ is independent of $t$, under some growth assumptions on the coefficients $q_{ij}$, $b_i$ and $c$ at infinity
and assuming that they are bounded in a small neighborhood of $\partial \Om$, we prove an uniform gradient estimate for $G_{\mathcal B}(t,s)f$. More precisely,
we show that for any $T>s \in I$, there exists
a positive constant $C_{s,T}$ such that
\begin{align}\label{graest_intro}
\|\nabla_x G_{\mathcal B}(t,s)f\|_{\infty} \le \frac{C_{s,T}}{\sqrt{t-s}}\|f\|_{\infty},\qquad\;\, t\in (s,T),
\end{align}
for any $f \in C_b(\Om)$. Estimate \eqref{graest_intro} (which can be then extended, by the evolution law, to all $t \in (s,+\infty)$) is classical when the coefficients of $\mathcal{A}(t)$ are bounded and $\Om$
is an open set with sufficiently smooth boundary, either bounded or unbounded (see \cite{Lun95Ana}).
Recently, it has been proved for the semigroup $T(t)$ associated in $C_b(\Omega)$
to autonomous elliptic operators with unbounded coefficients, both in the case of homogeneous
Neumann (first in convex sets \cite{BerFor04Gra} and, then,  in the general case \cite{BerForLor07Gra})
and Dirichlet boundary conditions  \cite{ForMetPri04Gra}.
Very recently, we proved estimate \eqref{graest_intro} for the solution to problem \eqref{NAPC} in $\Rd_+$
when homogeneous Dirichlet and Neumann boundary conditions are prescribed on $\partial\Rd_+$.
The simple geometry of $\Rd_+$ and suitable assumptions on the coefficients of the operator $\A(t)$,
allowed to extend these latter ones to $\Rd$ and to reduce the problem to the whole space $\Rd$, where gradient estimates were already known (\cite{KunLorLun09Non}). A symmetry argument was then used to come back to the Neumann and Dirichlet Cauchy problems set in $\Rd_+$.

In our situation the key tools to prove \eqref{graest_intro} are the Bernstein method,
the maximum principle in Proposition \ref{nsMP} and the geometric Lemma \ref{geometric} which allows to locally transform the boundary Cauchy problem \eqref{NAPC} into a Cauchy problem in the halfspace $\Rd_+$ where homogeneous Robin boundary conditions are prescribed. Bernstein method works very well in the whole space and it is easy to explain: one considers the function $t\mapsto v(t,\cdot)=(G(t,s)f)^2+a(t-s)|\nabla_xG(t,s)f|^2$ and shows that, under suitable assumptions
and a suitable choice of the positive parameter $a$, $D_tv-\A(t)v\le 0$. A variant of the maximum principle
reveals that the supremum of function $v$ is attained on $\{s\}\times\Rd$, and the gradient estimate follows at once.
When $\Rd$ is replaced by an open set $\Omega$, things become much more difficult. Indeed, the supremum of $v$
could be attained on $\partial\Omega$. Hence, one needs to bound the suprema of $v$ on $\partial\Omega$.
In the autonomous case, this has been done in the case of Dirichlet and Neumann boundary conditions.
In the first case an a priori gradient estimate on the boundary of $\Om$ has been proved by a comparison argument,
which reveals that therein the function $t\mapsto\sqrt{t}|\nabla_x T(t)f|$
can be bounded uniformly by a constant times the sup-norm of $f$. The argument in \cite{ForMetPri04Gra} can not be adapted
to the case of different boundary conditions. Neumann boundary conditions have been considered first in convex domains  (see \cite{BerFor04Gra}), where the geometry
of $\Omega$ shows that the normal derivative of
$|\nabla_x T(t)f|^2$ is nonpositive, so that the normal derivative of $v$ is nonpositive on $\partial\Omega$ as well and, consequently,
the supremum of $v$ is attained on $\{0\}\times\Omega$. When $\Omega$ is nonconvex,
the normal derivative of $|\nabla_x T(t)f|^2$
does not need to be nonnegative. But, as in \cite{BerForLor07Gra}, replacing $v$ by the function
$t\mapsto w(t,\cdot)=(T(t)f)^2+am|\nabla_x T(t)f|^2$ for a suitable function $m$, which takes into account
the curvatures of $\partial\Omega$, one can still prove that $D_tw-\A w$ and the normal
derivative of $w$, are nonnegative in $\Omega$ and $\partial\Omega$, respectively.

Clearly, for more general unbounded domains and more general boundary conditions, the same arguments do not work, therefore we need to develop new strategies
to prove the uniform gradient estimate \eqref{graest_intro}. Here, the idea is to use the regularity of the domain to go back by means of local charts
to problems defined in $\Rd_+$ or in $\Rd$. Assuming more smoothness on the domain $\Om$ and the vector $\beta$, we determine coordinate
transformations which, locally transform the homogeneous boundary condition ${\mathcal B}u=0$
on the boundary $\partial \Om$ to an homogeneous Robin boundary condition on $\R^{d-1}\times\{0\}$.
Thus, under the assumption that the coefficients of $\A(t)$ are bounded
only in a neighborhood of the boundary $\partial \Om$, we prove an uniform gradient estimates in a small strip $\Omega_{\delta}$ near the boundary.
Finally, some growth assumptions on the diffusion coefficients and the potential term and a quite standard dissipativity
condition on the drift term $b$, are enough to show that \eqref{graest_intro} is satisfied also in $\Omega\setminus\Omega_{\delta}$.
We point out that, differently from \cite{BerFor04Gra,BerForLor07Gra,ForMetPri04Gra}, we do not assume that the
diffusion coefficients $q_{ij}$ are globally bounded together with their spatial gradients. Moreover, our results seem to be new also in the autonomous case when ${\mathcal B}$ is a general
first-order boundary operator. In particular, we can cover also the case when $\gamma$ changes sign on $\partial\Omega$.

The special case when $\Om$ is convex and homogeneous Neumann boundary conditions are prescribed, can be treated and estimate \eqref{graest_intro} can be proved without assuming any additional smoothness assumption on the domain and any hypotheses of boundedness for the coefficients of $\A(t)$ in a neighborhood of the boundary. This can be done adapting the arguments used in the autonomous case, described here above.

Also when $\Om=\Rd_+$ and homogeneous Robin boundary conditions
are prescribed on $\R^{d-1}\times \{0\}$, we do not need to assume that the drift term $b$ and the potential term $c$ are bounded. Indeed,
a simple trick allows us to transform homogeneous Robin boundary condition into homogeneous Neumann condition on $\partial \Rd_+$. Hence, we are reduced to a problem set in a convex set with Robin boundary conditions, to which we can apply the already established results.

The paper is split into section as follows. In Section \ref{sect_main} we state the main
assumptions on the coefficients of the operators $\A(t)$ and $\B(t)$ and on the domain $\Om$,
recalling also some consequences of the smoothness of the domain. In Section \ref{sect_exi},
we first prove a maximum principle for solutions to the problem \eqref{NAPC}, which are continuous
in $\big([s,+\infty)\times \overline{\Om}\big)\setminus \big(\{s\}\times \partial \Om\big)$. Then, we
construct the solution to the problem \eqref{NAPC}. In Section \ref{sect-evol} we introduce the evolution
operator $G_{\mathcal B}(t,s)$ and we investigate
on some of its qualitative properties, such as compactness. Section \ref{sect_graest} is devoted to prove the uniform gradient
estimates \eqref{graest_intro} and in Section \ref{sect_ex} we provide some examples of operators
to which our results can be applied. The appendix collects some technical results used in the paper.
\medskip

\noindent
{\bf Notations.}
For any open set (or the closure of an open set)
${\mathcal O}$, any interval $J\subset\R$ and any $\delta>0$, we set
${\mathcal O}_{\delta}:=\{x\in \overline{{\mathcal O}}:\,r_{\mathcal O}(x)<\delta\}$ (where
$r_{\mathcal O}(x)={\rm dist}(x,\partial\mathcal O)$) and  ${\mathcal O}_J:=J\times {\mathcal O}$. Further,
by $\nu(x)$ we mean the outward unit normal to $\partial {\mathcal O}$ at $x$.

We assume that the reader is familiar with the spaces $C^k({\mathcal O})$ ($k\ge 0$)
and $C^{\alpha,\beta}({\mathcal O}_J)$ ($\alpha,\beta\ge 0$).
By $C^k_b({\mathcal O})$ we denote the subspace of $C^k({\mathcal O})$ consisting of functions which are bounded
together with all existing derivatives.
We use the
subscript ``$c$'' (resp. ``$0$'') for spaces of functions  with compact support
(resp. for spaces of functions vanishing on $\partial {\mathcal O}$ and at infinity).
When $k\in (0,1)$, we write $C^k_{\rm loc}({\mathcal O})$ to denote the space of all $f\in C({\mathcal O})$
which are H\"older continuous in any compact set of ${\mathcal O}$. Analogously, we define the spaces
$C^{\alpha/2,\alpha}_{\rm loc}({\mathcal O}_J)$ and $C^{1+\alpha/2,2+\alpha}_{\rm loc}({\mathcal O}_J)$ ($\alpha\in (0,1)$).

The notations $D_tf:=\frac{\partial f}{\partial t}$,
$D_if:=\frac{\partial f}{\partial x_i}$, $D_{ij}f:=\frac{\partial^2f}{\partial x_i\partial x_j}$
are extensively used, as well as the notation $J_xf$, to denote the Jacobian matrix, with respect to the spatial variables,
of the function $f:{\mathcal O}_J\to\Rd$.
$\chi_A$ denotes the characteristic function of the set $A\subset {\mathcal O}$ and $\one:=\chi_{\mathcal O}$.
The Euclidean ball with center at $x_0$ and radius $R>0$ is denoted by $B_{R}(x_0)$, $B_R:=B_R(0)$ and $B_R^+:=B_R\cap\Rd_+$.
Similarly, ${\mathcal O}^R$ denotes the set ${\mathcal O}\cap B_R$. Occasionally,
we find it convenient to split $\Rd\ni x=(x',x_d)$ with $x_d\in\R$. Finally, $a^+:=\max\{a,0\}$ for any $a\in\R$.

\section{Main assumptions and preliminaries}\label{sect_main}

Let $I\subset \R$ be an open right halfline (possibly $I=\R$) and $\Om$ be a domain of
$\R^d$. Let us introduce our standing assumptions on the domain $\Omega$ and on the coefficients of the operators $\A(t)$ in
\eqref{operator}:
\begin{hyp}\label{hyp1}
\begin{enumerate}[\rm (i)]
\item
$\partial \Om$ is uniformly of class $C^{2+\alpha}$ for some $0<\alpha<1$;
\item
$q_{ij}$, $b_{i}$ and $c$
belong to $C^{\alpha/2,\alpha}_{\rm{loc}}(\overline{\Om_I})$ for every $i,j=1,\dots,d$;
\item
$c_0:= \inf_{\Om_I}c\ge 0$;
\item
$Q$ is uniformly elliptic, i.e., for every $(t,x)\in \Om_I$, the matrix $Q(t,x)$
is symmetric and there exists a function $\eta:\Om_I\to\R^+$ such that
$0<\eta_0:=\inf_{\Om_I}\eta$ and
$\langle Q(t,x)\xi,\xi\rangle\geq\eta(t,x)|\xi|^2$ for any $\xi\in \R^d$ and $(t,x)\in \Om_I$.
\end{enumerate}
\end{hyp}
\begin{rmk}\label{local charts}{\rm
\begin{enumerate}[\rm (a)]
\item
Hypothesis \ref{hyp1}(i) is standard when problems are defined on unbounded domains. It means that
\begin{enumerate}[\rm (i)]
\item
there exist $R>0$, a (at most countable) collection of open balls $B_R(x_h)=:V_h$, $h \in \N$, covering $\partial \Om$, and $k\in\N$
such that $\sum_{h=1}^{+\infty}\chi_{V_h}\le k$ in $\R^d$, i.e., $\bigcap_{h\in H}V_h= \varnothing$ if $H\subset \N$ contains more than $k$ elements;
\item
there exist coordinate transformations  $\psi_h:V_h\to B_1$ ($h \in \N$), which are $C^{2+\alpha}$-diffeomorphisms such that
$\psi_h(V_h\cap \Om)=B_1^+$, $\psi_h(V_h\cap \partial \Om)=B_1\cap\partial\Rd_+$ for each $h$,
and
$\sup_{h \in \N}\left(\|\psi_h\|_{C^{2+\alpha}(V_h)}+\|\psi_h^{-1}\|_{C^{2+\alpha}(V_h)}\right)<+\infty$;
\item
there exists $\varepsilon>0$ such that $\bigcup_{h\in\N}B_{R/2}(x_h)\supset\Omega_{\varepsilon}$.
\end{enumerate}
\item
The smoothness of $\partial\Omega$ implies that the distance function $r_{\Omega}$
belongs to $C^2_b(\Om_\delta)$ for some $\delta>0$.
For any $x \in \Om_\delta$, it holds that $\nabla r_{\Omega}(x)= -\nu(\pi(x))$, where $\pi(x)$ is the projection of $x$ on
$\partial \Om$. Finally, the equiboundedness of the $C^{2+\alpha}$-norms of $\psi_h$ and $\psi_h^{-1}$, shows that
$\kappa= \inf_{x \in \partial \Om}\left\{\langle J \nu(x) \tau, \tau\rangle:\,\,|\tau|=1,\,\langle\tau, \nu(x)\rangle =0\right\}\in \R$.
\item
Since the last component $\psi_h^d$ of the function $\psi_h$ ($h\in\N$) identically vanishes on
$\partial\Omega\cap B_R(x_h)$ and it is positive inside $\Omega\cap V_h$, $\nabla\psi_h^d=-|\nabla_x\psi_h^d|\nu$ in $\partial\Omega\cap B_R(x_h)$.
\end{enumerate}
}\end{rmk}

As far as the boundary operators ${\mathcal B}(t)$ in \eqref{OpB} are concerned, when $\beta\equiv 0$, we assume that $\gamma\equiv 1$
in order to recover the Cauchy Dirichlet problem. On the other hand when $\beta\not\equiv 0$, we assume the following
assumptions on the coefficients of ${\mathcal B}(t)$.
\begin{hyp}\label{hyp1-bc}
\begin{enumerate}[\rm (i)]
\item
$\beta_i$ $(i=1,\ldots,d)$ and $\gamma$ belong to $C^{(1+\alpha)/2,1+\alpha}_{\rm loc}(\overline{I}\times \partial\Omega)$;
\item
$\gamma$ is bounded from below and $|\beta|\equiv 1$ in $I\times \partial \Om$;
\item
$\inf_{(t,x)\in [a,b]\times\partial\Omega}\langle\beta(t,x),\nu(x)\rangle>0$ for any $[a,b]\subset I$.
\end{enumerate}
\end{hyp}

To guarantee the uniqueness of the bounded classical solution to the problem \eqref{NAPC}
(see Definition \ref{defi}), we assume the following condition.
\begin{hyp}
\label{hyp-3}
\begin{enumerate}[\rm (i)]
\item
For any bounded interval $J\subset I$ there exist a positive function
$\varphi=\varphi_J\in C^2(\overline{\Om_J})$
and a positive number $\lambda=\lambda_J$ such that
$\varphi$ blows up as $|x|\to +\infty$, uniformly with respect to $t\in J$,
and $D_t \varphi-\mathcal{A}\varphi+\lambda\varphi> 0$ in $\Om_J$.
\item
When $\beta\not\equiv 0$, we require in addition that
${\mathcal B}\varphi\ge 0$ in $J\times\partial\Omega$.
\end{enumerate}
\end{hyp}

\begin{rmk}
\label{rem-2.5}
{\rm Actually, the condition on the sign of $c_0$ is not restrictive; Hypotheses \ref{hyp1}(iii) can be replaced by the assumption that $c_0>-\infty$.
Indeed, if $c_0<0$, and $u$ solves problem \eqref{NAPC} then the function $(t,x)\mapsto \tilde{u}(t,x)= e^{c_0(t-s)}u(t,x)$, which has the same regularity as $u$,
satisfies $D_t\tilde{u}-\A_0\tilde{u}=0$, where $\A_0 u= \A u+ c_0 u$ has a nonnegative zero-order coefficient.
Moreover, $\A_0$ satisfies Hypotheses \ref{hyp-3} with the same
Lyapunov function $\varphi$ and the same positive constant $\lambda$.}
\end{rmk}

\section{Existence and uniqueness}
\label{sect_exi}
Here, we prove existence and uniqueness of the bounded classical solution to problem \eqref{NAPC}. Throughout this section, we denote by $S$ the
set $\{s\}\times\partial\Omega$.
\begin{defi}
\label{defi}
A function $u$ is called a  bounded classical solution
of the problem \eqref{NAPC} if
$u\in C^{1,2}(\Omega_{(s,+\infty)})\cap C_b(\overline{\Omega_{[s,+\infty)}}\setminus S)$ and satisfies \eqref{NAPC}.
\end{defi}

\subsection{The case when $\gamma\ge 0$}
\label{subsect-3.1}
The uniqueness of the classical solution to problem \eqref{NAPC} is a consequence of suitable maximum principle.

\begin{prop}\label{nsMP}
Let $T>s\in I$ and $u\in C^{1,2}(\Om_{(s,T)})\cap C_b(\overline{\Om_{(s,T)}}\setminus S)$
satisfy
\begin{equation}
\left\{
\begin{array}{ll}
D_t u(t,x)- (\A u)(t,x)\leq 0, \quad& (t,x)\in \Om_{(s,T)},\\[1mm]
({\mathcal B}u)(t,x)\leq 0, \quad & (t,x)\in (s,T)\times\partial\Omega,\\[1mm]
u(s,x)\le 0, & x\in\Omega.
\end{array}
\right.
\label{studente}
\end{equation}
Then, $u\leq 0$ in  $\Om_{(s,T)}$.
\end{prop}
\begin{proof}
Let $\lambda=\lambda_{[s,T]}$ and $\varphi=\varphi_{[s,T]}$ be the constant and the function in Hypothesis \ref{hyp-3}.
Up to replacing $\lambda$ with a larger value, if needed, we can assume that $D_t\varphi-\A\varphi+\lambda\varphi>0$ in $\Omega_{(s,T)}$.
 To prove that $u\le 0$ in $\Omega_{(s,T)}$, for any $n\in\N$ we introduce the function $v_n$, defined by
$v_n(t,x)=e^{-\lambda(t-s)}u(t,x)-n^{-1}\varphi(t,x)$ for any $(t,x)\in\overline{\Omega_{(s,T)}}\setminus S$,
and prove that $v_n$ is nonpositive. Then, letting $n\to +\infty$ we conclude that $u$ is nonpositive
as well.

Since $\varphi$ tends to $+\infty$ as $|x|\to +\infty$, uniformly with respect to $t\in [s,T]$, and $u$ is bounded,
$v_n$ tends to $-\infty$ as $|x|\to +\infty$, uniformly with respect to $t\in [s,T]$, for any $n\in\N$.
We can thus fix $R>0$ large enough such that $v_n<0$ in $[s,T]\times (\Omega\setminus B_R)$.
It thus follows that we just need to prove that $v_n\le 0$ in $(s,T]\times\Omega^R$.

We split the rest of the proof in two steps. In the first one, we assume that $u$ is continuous in the whole of $\overline{\Omega_{(s,T)}}$.
Then, in Step 2, we consider the general case.

{\em Step 1.} Since $u$ is continuous in $\overline{\Omega_{(s,T)}}$, $v_n$ satisfies
\begin{equation}
\label{1}
\left\{
\begin{array}{ll}
D_tv_n(t,x)-({\mathcal A}v_n)(t,x)+\lambda v_n(t,x)< 0,\qquad\;\,&(t,x)\in (s,T]\times\Omega^R,\\[1mm]
({\mathcal B}v_n)(t,x)\le 0,\qquad\;\,&(t,x)\in (s,T]\times\partial_1\Omega^R,\\[1mm]
v_n(t,x)< 0,\qquad\;\,&(t,x)\in (s,T]\times\partial_2\Omega^R,\\[1mm]
v_n(s,x)<0,\qquad\;\,& x\in\overline{\Omega^R},
\end{array}
\right.
\end{equation}
where $\partial\Omega_R=\partial_1\Omega^R\cup \partial_2\Omega^R:=(\partial\Omega\cap B_R)\cup (\overline\Omega\cap\partial B_R)$.
We follow the lines in the proof of \cite[Thm. 2.16]{Fri64Par}. For this purpose, we introduce the set
${\mathcal J}_n:=\{r\in [s,T]: v_n<0 \textrm{ in } \overline{\Omega^R_{(s,r)}}\}$,
which contains $s$. Since $u$ and $\varphi$ are continuous in $\overline{\Omega_{(s,T)}}$,
the function $v_n$ is uniformly continuous in $\overline{\Omega^R_{(s,T)}}$.
This implies that ${\mathcal J}_n$ is an interval and $\sup {\mathcal J}_n>s$.
Let us denote by $\tau_n$ the supremum of ${\mathcal J}_n$ and prove that $\tau_n=T$.
By contradiction, we assume that $\tau_n<T$.
Then, by continuity $v_n(\tau_n,\cdot)\le 0$ in $\overline{\Omega^R}$ and there exists $x_n\in\overline{\Omega^R}$
such that $v_n(\tau_n,x_n)=0$.
The point $(\tau_n,x_n)$ turns out to be the maximum point of the restriction of $v_n$ to
$\overline{\Omega^R_{(s,\tau_n)}}$. Moreover, $x_n$ can not belong to $\Omega^R$, otherwise we would have
$({\mathcal A}v_n)(\tau_n,x_n)-\lambda v_n(\tau_n,x_n)\le 0$ and $D_tv_n(\tau_n,x_n)\ge 0$, thus
contradicting \eqref{1}.
Hence, $x_n\in \partial \Om^R$. Actually, $x_n$ can not belong
to $\partial_2\Omega^R$ and, clearly, it can not belong to $\partial_1\Omega^R$, if ${\mathcal B}\equiv I$. Indeed,
in this case ${\mathcal B}v_n=-n^{-1}\varphi$, which is negative (see Hypothesis \ref{hyp-3}(i)).
On the other hand, if ${\mathcal B}$ is a first-order boundary operator and $x_n\in\partial_1\Omega_R\subset\partial\Omega$, then we would have
$\langle\beta(\tau_n,x_n),\nabla_x v_n(\tau_n,x_n)\rangle >0$,
since at each point of $\partial\Omega$ the interior sphere condition is satisfied (see e.g., \cite[Thm. 3.7]{protter}).
But this contradicts the boundary condition in \eqref{1}.
We thus conclude that $\tau=T$ so that $v_n$ is negative in $\overline{\Omega^R_{(s,T)}}$.

\emph{Step 2.} We now consider the general case when $u$ is not continuous on $S$. Since the above arguments do not work, we use a different strategy
and we adapt to our situation an idea which has been already used in \cite[Thm. A.2]{ForMetPri04Gra} in the case
of autonomous Dirichlet Cauchy problems. For any $n\in\N$, we introduce the function
$w_{n,\varsigma}=v_n-M_n\varsigma^{\varepsilon\eta_0}\psi_{\varsigma}$, where $M_n=\sup_{\Omega_{(s,T)}^R}v_n$,
\begin{align*}
\psi_{\varsigma}(t,x)=\frac{1}{(t+\varsigma-s)^{\varepsilon\eta_0}}\exp\left (t+\varsigma-s-\frac{\varepsilon
\omega^2(x)}{t+\varsigma-s}\right ),\qquad\;\, (t,x)\in \Om_{(s,T)},
\end{align*}
$\omega=\vartheta r_{\Omega}+1-\vartheta$, where $\vartheta\in C^2_b(\overline\Omega)$ satisfies
$\chi_{\Omega_{\delta/2}}\le\vartheta\le\chi_{\Omega_{\delta}}$ and $\delta$ is sufficiently small to have
$r_{\Rd\setminus\Omega}\in C^2_b(\overline{\Omega_{\delta}})$ (see Remark \ref{local charts}).
Finally, $\varsigma$ and $\varepsilon$ are positive parameters. Function $w_{n,\varsigma}$ has the same regularity as $u$. We claim that
\begin{enumerate}[\rm (i)]
\item
$D_tw_{n,\varsigma}-\A w_{n,\varsigma}+\lambda w_{n,\varsigma}\le 0$ in $\Omega^R_{(s,T)}$ and ${\mathcal B}w_{n,\varsigma}\le 0$ in
$(s,T]\times\partial\Omega$, for suitably fixed $\varepsilon>0$ and any $\varsigma\in (0,1)$;
\item
there exists $\tau=\tau(\varsigma)$ such that $w_{n,\varsigma}\le 0$ in $\Omega^R_{(s,s+\tau]}$.
\end{enumerate}
Since $w_{n,\varsigma}$ is continuous in $\overline{\Omega_{(s+\tau,T)}}$, we can then apply Step 1 to show that $w_{n,\varsigma}\le 0$
in $\overline{\Omega^R_{(s+\tau,T)}}$ and we conclude that $w_{n,\varsigma}\le 0$ in $\Omega^R_{(s,T)}$. Letting $\varsigma\to 0^+$ we deduce
that $v_n\le 0$ in $\Omega^R_{(s,T)}$.

To check property (i), we prove that there exists $\varepsilon>0$ such that, for any $\varsigma\in (0,1)$,
$D_t\psi_{\varsigma}-\A\psi_{\varsigma}+\lambda\psi_{\varsigma}$ and ${\mathcal B}\psi_{\varsigma}$ are positive in $\Omega^R_{(s,T)}$
and in $(s,T]\times\partial\Omega$, respectively.
Observe that $D_t\psi_{\varsigma}-\A\psi_{\varsigma}+\lambda\psi_{\varsigma}$ is positive in $\Omega^R_{(s,T)}$ if and only if the function
$h_{\varsigma}$, defined by
\begin{align*}
h_{\varsigma}(t,\cdot)=&(\lambda+1+c(t,\cdot))(t+\varsigma-s)^2-\varepsilon\eta_0 (t+\varsigma-s)
- 4\varepsilon^2 \omega^2 \langle Q(t,\cdot)\nabla\omega, \nabla\omega\rangle\\
&+ \varepsilon \omega^2+ 2\varepsilon (t+\varsigma-s) \langle Q(t,\cdot)\nabla\omega, \nabla\omega\rangle
+2\varepsilon (t+\varsigma-s)\omega(\A+c)\omega,
\end{align*}
for any $t\in (s,T)$, is positive in $\Omega^R_{(s,T)}$. To estimate the sign of the function $h_{\varsigma}$, we denote by $K_0$ and $K_1$
the supremum over $\Omega^R_{(s,T)}$ of the functions $\langle Q\nabla\omega,\nabla\omega \rangle$ and
$\langle b, \nabla\omega\rangle +\textrm{Tr}(Q D^2 \omega)$, respectively, and observe that
$\Omega= \{x\in\Omega: |\nabla\omega(x)|>3/4\}\cup\{x\in\Omega: \omega(x)\ge\sigma\}=:A \cup B$,
for a suitable $\sigma>0$. Indeed, for any $x\in\Omega$,
$|\nabla \omega(x)|\ge |\nabla \omega(\xi)|-|\nabla \omega(x)-\nabla \omega(\xi)|\ge 1 -\|\omega\|_{C^2_b(\Omega_\delta)}|x-\xi|$,
where $\xi$ is the unique projection of $x$ on $\partial \Om$. Hence, $|\nabla\omega|\ge \frac{3}{4}$ in
$\Omega_{\sigma}$ if $\sigma=(4\|\omega\|_{C^2_b(\Omega_\delta)})^{-1}$. Now, using the inequalities
\begin{align*}
&2\varepsilon (t+\varsigma-s)
\omega \left (\langle b(t,\cdot), \nabla\omega\rangle +\textrm{Tr}(Q(t,\cdot) D^2 \omega)\right)\ge
-\varepsilon^{\frac{3}{2}}\omega^2-(t+\varsigma-s)^2 K_1^2\sqrt{\varepsilon},\\[1mm]
&\langle Q\nabla\omega,\nabla\omega\rangle\ge \frac{9}{16}\eta_0\chi_A,\\[1mm]
&-\varepsilon (t+\varsigma-s)=-\varepsilon (t+\varsigma-s)\chi_A-\varepsilon (t+\varsigma-s)\chi_B\\
&\qquad\qquad\qquad\;\,\ge -\varepsilon (t+\varsigma-s)\chi_A-\frac{1}{2}\varepsilon^{\frac{3}{2}}\chi_B
-\frac{1}{2}(t+\varsigma-s)^2\sqrt{\varepsilon}\chi_B,
\end{align*}
and, recalling that $c\ge 0$, we can estimate
\begin{align*}
h_{\varsigma}\ge\! \left (\hskip -1truemm\lambda+1-\frac{1}{2}\sqrt{\varepsilon}\chi_B-K_1^2\sqrt{\varepsilon}\hskip -.2truemm\right )\hskip -.5truemm (\cdot+\varsigma-s)^2
+\varepsilon\omega^2[1-\sqrt{\varepsilon}(1+4K_0\sqrt{\varepsilon})]-\frac{1}{2}\varepsilon^{\frac{3}{2}}\chi_B.
\end{align*}
It is now clear that, if $\varepsilon\le\varepsilon_0:=\min\{(\lambda+1)^2K_1^{-4},(8K_0)^{-2}(-1+\sqrt{1+16K_0})^2\}$, then
$h_{\varsigma}(t,\cdot)$ is nonnegative
in $A$. On the other hand, $\omega\ge\sigma$ in $B$. Therefore, if $\varepsilon\le (8K_0)^{-2}(-1+\sqrt{1+8K_0})^2$,
it holds that
\begin{align*}
\varepsilon\omega^2[1-\sqrt{\varepsilon}(1+4K_0\sqrt{\varepsilon})]-\frac{1}{2}\varepsilon^{\frac{3}{2}}
\le\frac{1}{2}\varepsilon (\sigma^2-\sqrt{\varepsilon}).
\end{align*}
It is now clear that, if $\varepsilon\le\varepsilon_1:=\min\{(8K_0)^{-2}(-1+\sqrt{1+8K_0})^2, 4(\lambda+1)^2(1+2K_1^2)^{-2},\sigma^4\}$,
then $h_{\varsigma}(t,\cdot)\ge 0$ in $B$. Taking $\varepsilon=\varepsilon_1<\varepsilon_0$, it follows that $h_{\varsigma}\ge 0$
in $\Omega_{(s,T)}$ for any $\varsigma\in (0,1)$. To complete the proof of
the claim, we observe that
\begin{align*}
({\mathcal B}\psi_{\varsigma})(t,x)
=\left (\frac{2\omega(x)\varepsilon_1}{t+\varsigma-s}\langle \beta(t,x),\nu(x)\rangle+\gamma(t,x)\right )\psi_{\varsigma}(t,x)>0,
\end{align*}
for any $(t,x)\in [s,T]\times\partial\Omega$ and $\varsigma\in (0,1)$, since $\nabla\omega\equiv -\nu$ on $\partial\Omega$ and
$\langle\beta,\nu\rangle\ge 0$ in $[s,T]\times\partial\Omega$ by Hypothesis \ref{hyp1}(iii). The claim is now proved.

Finally, to check property (ii), we set $\Omega^{R,\eta}:=\{x\in\Omega^R: \omega(x)\le\eta\}$ and split
$\Omega=\Omega^{R,\eta}\cup (\Omega\setminus\Omega^{R,\eta})$,
where $\eta=\eta(\varsigma)>0$ satisfies $\varsigma -\frac{\varepsilon \eta^2}{\varsigma}>0$. It follows that
$e^{\varsigma-\frac{\varepsilon\omega^2}{\varsigma}}>1$ in $\Om^{R,\eta}$ and, by continuity we can find $\tau=\tau(\varsigma)>0$ such that
\begin{align*}
\varsigma^{\varepsilon\eta_0}\psi_{\varsigma}(t,x)
=\frac{\varsigma^{\varepsilon\eta_0}e^{t+\varsigma-s}}{(t+\varsigma-s)^{\varepsilon\eta_0}}e^{-\frac{\varepsilon\omega^2(x)}{t+\varsigma-s}}
>1,\qquad\;\,(t,x)\in [s,s+\tau]\times \Om^{R,\eta}.
\end{align*}
This implies that $w_{n,\varsigma}< v_n-M_n\le 0$ in $([s,s+\tau]\times \Om^{R,\eta})\setminus S$. Moreover, since $w_{n,\varsigma}(s,\cdot)\leq 0$ in
$\Om^R\setminus\Omega^{R,\eta}$ and $w_{n,\varsigma}$ is continuous in $[s,T]\times \overline{\Omega^R\setminus\Omega^{R,\eta}}$, up to
replacing $\tau$ with a smaller value if needed, we can assume that $w_{n,\varsigma}(t,x)\leq 0$ for any $(t,x)\in ([s,s+\tau]\times \Om^R)\setminus S$.
Property (ii) follows.
\end{proof}

In order to get existence of a unique solution to the problem \eqref{NAPC} we proceed by steps. In the following proposition we consider the case
when the datum $f$ vanishes at infinity and on the boundary of $\Omega$. We recall that $c_0$ is the infimum of the potential $c$
(see Hypothesis \ref{hyp1}(iii)).

\begin{prop}\label{smoothdatum}
For any $f\in C_0(\Omega)$, the Cauchy problem \eqref{NAPC} admits a unique bounded classical solution $u$.
It belongs to $C^{1+\alpha/2,2+\alpha}_{\rm{loc}}(\overline{\Omega}_{(s,+\infty)})\cap C_b(\overline{\Omega_{(s,+\infty)}})$ and satisfies the estimate
\begin{equation}
\label{estinf}
\| u(t,\cdot)\|_\infty\leq e^{-c_0(t-s)}\| f\|_\infty, \qquad\;\, t\geq s.
\end{equation}
If, further, $f\in C^{2+\alpha}_c(\Om)$, then $u \in C^{1+\alpha/2,2+\alpha}_{\rm{loc}}(\overline{\Om_{(s,+\infty)}})$.
\end{prop}
\begin{proof}
Uniqueness follows immediately by applying Proposition \ref{nsMP}. Estimate \eqref{estinf} can
be obtained applying the same proposition to the functions $\pm e^{c_0(t-s)}u(t,x)-\| f\|_\infty$
which satisfy \eqref{studente} with $t=+\infty$ and $\A$ being replaced by $\A+c_0$.

To prove the existence part, we first consider $f\in C^{2+\alpha}_c(\Omega)$ and use an approximation argument. For any $n\in\N$,
let $\vartheta_n$ be any smooth function such that $\chi_{B_n}\le\vartheta_n\le\chi_{B_{2n}}$. Moreover,let the functions $\mu_i\in C^{1+\alpha}_b(\Omega)$ ($i=1,\ldots,d$) satisfy $\mu_i=-D_ir_{\Omega}$ in a neighborhood of $\partial\Omega$.
We then approximate the coefficients $q_{ij}$, $b_j$, $c$, $\beta_i$ and $\gamma$ ($i,j=1,\ldots,d$)
by the (bounded) coefficients  $q_{ij}^{(n)}$, $b_j^{(n)}$, $c^{(n)}$, $\beta_i^{(n)}$ and $\gamma^{(n)}$, defined by
$q_{ij}^{(n)}=\vartheta_n q_{ij}+(1-\vartheta_n)\delta_{ij}$, $b_j^{(n)}=\vartheta_nb_j$, $c^{(n)}=\vartheta_nc$,
$\beta_i^{(n)}=\vartheta_n\beta_i+(1-\vartheta_n)\mu_i$, and $\gamma^{(n)}=\vartheta_n\gamma$, for any $i,j=1,\ldots,d$.
Clearly, $q_{ij}^{(n)}$, $b_j^{(n)}$, $c^{(n)}$ converge to $q_{ij}$, $b_j$, $c$, respectively,
locally uniformly in $\overline{\Omega_I}$.

Let ${\mathcal A}^{(n)}$ and ${\mathcal B}^{(n)}$ be, respectively, the differential operators defined as $\A$ and ${\mathcal B}$
with $(q_{ij},b_j,c,\beta,\gamma)$ being replaced by $(q_{ij}^{(n)},b_j^{(n)},c^{(n)},\beta^{(n)},\gamma^{(n)})$.

By \cite[Thms. IV.5.2 \& IV.5.3]{LadSolUra68Lin}, for any $n\in\N$, the Cauchy problem
\begin{equation}
\label{p_n}
\left\{
\begin{array}{lll}
D_t v(t,x)=(\A^{(n)}v)(t,x),\quad & (t,x)\in (s,+\infty)\times\Om,\\[1mm]
({\mathcal B}^{(n)}v)(t,x)= 0, \quad & (t,x)\in (s,+\infty)\times\partial\Om,\\[1mm]
v(s,x)= f(x), \quad & x\in \Om,
\end{array}
\right.
\end{equation}
admits a unique solution $u_n\in C_{\rm loc}^{1+\alpha/2,2+\alpha}(\overline{\Omega_{(s,+\infty)}})$. Moreover, for any
$m,n\in\N$, with $n>m$, the local Schauder estimates (see \cite[Thm. IV.10.1]{LadSolUra68Lin}) show that there exists a positive constant $c_m$,
independent of $n$, such that
$\| u_n\|_{C^{1+\alpha/2,2+\alpha}(\Om^m_{(s, s+m)})}\leq c_m \| f\|_{C^{2+\alpha}_b(\Om)}$. Applying
Arzel\`a-Ascoli theorem, we can determine a subsequence $(u_n^{(m)})$ converging in
$C^{1,2}(\overline{\Om^m_{(s,s+m)}})$ to a function $u^{(m)}\in C^{1+\alpha/2, 2+\alpha}(\Om^m_{(s,s+m)})$
which satisfies the equation $D_tu^{(m)}=\A u^{(m)}$ in $\Om^m_{(s,s+m)}$. Moreover, $u^{(m)}(s,\cdot)\equiv f$ in $\Om^m$ and
${\mathcal B}^{(n)}u^{(m)}\equiv {\mathcal B}u^{(m)}\equiv 0$ in $(s,s+m)\times (\partial\Omega\cap B_m)$. Since, without loss of generality,
we can assume that $(u_n^{(m)})\subset (u_n^{(m-1)})$ for any $m\in\N$, we can define the function $u:\overline{\Omega_{(s,+\infty)}}\to\R$
by setting $u(t,x)=u^{(m)}(t,x)$ for every $(t,x)\in \Om^m_{(s,s+m)}$ and every $m\in\N$; it belongs to
$C^{1+\alpha/2, 2+\alpha}_{\rm loc}(\overline{\Omega_{(s,+\infty)}})$ and satisfies \eqref{NAPC}.

Finally, we consider the general case when $f\in C_0(\Omega)$. We fix a sequence  $(f_n)\subset C^{2+\alpha}_c(\Omega)$
converging to $f$ uniformly in $\overline\Omega$ and denote by $u_n$ the unique bounded classical solution to the Cauchy problem \eqref{NAPC},
with $f$ being replaced by $f_n$. Applying estimate \eqref{estinf} to the function $u_n-u_m$, we obtain that $(u_n)$ is a Cauchy sequence
in $\Omega_{(s,T)}$ for any $T>s$. Hence, by the arbitrariness of $T>s$, $u_n$ converges uniformly in $\overline{\Omega_{(s,+\infty)}}$ to a function
$u\in C_b(\overline{\Omega_{(s,+\infty)}})$ which satisfies $u(s,\cdot)=f$.

To prove that $u$ is smooth, solves the differential equation and satisfies the boundary condition in \eqref{operator},
we apply a compactness argument, as in the first part of the proof, starting from the interior Schauder estimates
(see e.g., \cite[Thm. IV.10.1]{LadSolUra68Lin}) which show that the $C^{1+\alpha/2,2+\alpha}$-norm of the sequence $(u_n)$,
in any compact set of $(s,T]\times\overline{\Omega}$, is bounded by a constant independent of $n$.
Hence, there exists a subsequence $(u_{n_k})$ which converges locally uniformly, together with its derivatives,
to a function $v\in C^{1+\alpha/2,2+\alpha}_{\rm loc}(\overline{\Omega}_{(s,+\infty)})$ which satisfies the differential equation
$D_tv={\mathcal A}v$ in $\Om_{(s,+\infty)}$ and the boundary condition ${\mathcal B}v=0$ in $(s,+\infty)\times\partial \Om$.
Since, clearly, $v\equiv u$, $u$ is the bounded classical solution t the problem \eqref{NAPC}.
\end{proof}

We can now address the general case when $f\in C_b(\Omega)$.

\begin{thm}\label{existence}
For any $f \in C_b(\Om)$, problem \eqref{NAPC} has a unique bounded classical solution $u$.
Moreover, $u\in C^{1+\alpha/2,2+\alpha}_{\rm loc}(\overline{\Omega}_{(s,+\infty)})$, estimate \eqref{estinf} holds and, if $f\ge 0$
does not identically vanish then, $u$ is strictly positive in $\Omega_{(s,+\infty)}$.
\end{thm}

\begin{proof}
The uniqueness part and estimate \eqref{estinf} are consequences of Proposition \ref{nsMP}.

Let us prove the existence part. We fix $f \in C_b(\Om)$ and a sequence $(f_n)\subset C^{2+\alpha}_c(\Om)$ converging to $f$ uniformly
on compact subsets of $\Om$ and such that $\| f_n\|_\infty\leq\| f\|_\infty$. We denote by $u_n$ the bounded classical solution
to problem \eqref{NAPC} with $f$ being replaced by $f_n$.
Using the interior Schauder estimates as in the last part of the proof of Proposition \ref{smoothdatum}, we can prove that, up to a subsequence,
$u_n$ converges locally uniformly, together with its derivatives, to a function
$u\in C^{1+\alpha/2,2+\alpha}_{\rm loc}(\overline{\Omega}_{(s,+\infty)})$ which satisfies the differential equation $D_tu=\A u$ in
$\Omega_{(s,+\infty)}$ and the boundary condition in $(s,+\infty)\times\partial\Omega$.
So, to prove that $u$ solves problem \eqref{NAPC} we have to show that $u$ can be extended by continuity up to $t=s$ where it equals $f$.
We fix a compact set $K\subset \Om$ and a smooth and compactly supported function $\psi$
such that $0\leq \psi \leq 1$ and $\psi\equiv 1$ in $K$. Since $\psi f_n$ and $(1-\psi)f_n$ are compactly supported in $\Omega$ for every $n\in\N$,
by linearity and Proposition \ref{nsMP} we conclude that
$u_{f_n}=u_{\psi f_n}+ u_{(1-\psi)f_n}$, where $u_g$ denotes the unique bounded classical solution to problem \eqref{NAPC} with $f=g$.
Applying Proposition \ref{nsMP} to the functions $\pm u_{(1-\psi)f_n}-\| f\|_\infty( 1-u_{\psi})$, we get the estimate
$\|u_{(1-\psi)f_n}\|_\infty\leq\| f\|_\infty( 1-u_{\psi})$.
Hence, $|u_{f_n}- f| \leq |u_{\psi f_n}-\psi f|+\| f\|_{\infty}( 1-u_{\psi})$ in $\Omega_{(s,+\infty)}$.
Letting $n\to +\infty$ in the previous inequality, from the last part of the proof of Proposition \ref{smoothdatum} we obtain
$|u- f| \leq |u_{\psi f}-\psi f| +  \|f\|_{\infty}(1-u_{\psi})$ in $\Omega_{(s,+\infty)}$.
Since $\psi\equiv 1$ in $K$, it now follows that $u$ can be extended by continuity at $t=s$ by setting $u(s,\cdot)=f$ in $K$.
By the arbitrariness of $K$ we deduce that $u=u_f$.

Finally let us prove that $u$ is positive in $\Om_{(s,+\infty)}$ if $f\ge 0$ is not identically zero. Proposition \ref{nsMP}
shows that $u\ge 0$ in $\Om_{(s,+\infty)}$. By contradiction, let us assume that there exists $(t_0, x_0) \in \Om_{(s,+\infty)}$
such that $u(t_0, x_0)=0$. Let us consider an open set $\Om^*\ni x_0$, compactly contained in $\Om$, where $f$ does not identically vanish.
By applying \cite[Thm. 3.7]{protter}  to $-u$ in the cylinder $\Om^*_{(s,t_0)}$ we deduce that $u(t,x)=0$ for $(t,x)\in \times \Om^*_{[s,t_0]}$
getting to a contradiction.
\end{proof}

\begin{rmk}
\label{chiusura}
{\rm If the coefficients of the first-order operator ${\mathcal B}$ are independent of $t$ and
belong to $C^{1+\alpha}_b(\partial\Omega)$, then, for any $f\in C_b(\overline\Omega)$, the classical solution $u_f$ to problem \eqref{NAPC} is continuous in the whole of $\overline{\Omega_{(s,+\infty)}}$. Indeed, under these
conditions, we can repeat the same arguments as in the proof of Proposition $\ref{smoothdatum}$ (without approximating the boundary
operator ${\mathcal B}$) to show that, for any $g\in C^{2+\alpha}_b(\Omega)$ such that $\mathcal Bg\equiv 0$ on $\partial \Om$, the solution $u_g$ to problem \eqref{NAPC} belongs to $C^{(1+\alpha)/2,2+\alpha}_{\rm loc}(\overline{\Omega_{(s,T)}})$.
Since any function $f\in C_b(\overline{\Omega})$, which vanishes at $\infty$, is the uniform limit of a sequence of functions $(g_n)$
in $C^{2+\alpha}_b(\Omega)$, which vanish at infinity and satisfy ${\mathcal B}g_n\equiv 0$ on $\partial\Omega$ for any $n\in\N$,
arguing as in the proof of Proposition \ref{smoothdatum}, we conclude that $u_f$ is continuous in $\overline{\Omega_{(s,+\infty)}}$
for any $g$ as above.
As a by product, in the proof of Theorem \ref{existence} we can take as $K$ a compact subset of $\overline{\Omega}$. Since
$u_{f_n\psi}$ converges to $u_{f\psi}$ uniformly in $\Omega_{(s,+\infty)}$, we can estimate
$|u_f-f|\le |u_{\psi f}-\psi f|+\|f\|_{\infty}(1-u_{\psi})$. This inequality shows that $u_f$ is continuous on $\{s\}\times K$.
The arbitrariness of $K$ yields the continuity of $u_f$ on $\{s\}\times\overline{\Omega}$ and, consequently, in
$\overline{\Omega_{(s,+\infty)}}$.}
\end{rmk}

\subsection{The general case}

We now consider the general case when $\gamma$ can assume also negative values. We stress that the arguments used in the previous subsection
to prove uniqueness and estimate \eqref{estinf} fail.
To overcome these difficulties we assume an additional assumption.

\begin{hyp0}
\label{ipofisi}
There exist a function $\phi\in C^{2+\alpha}_{\rm loc}(\overline\Omega)\cap C_b(\overline\Omega)$, with positive infimum, and a constant $H$ such that
${\mathcal A}\phi\le H\phi$ in $\Omega_I$  and ${\mathcal B}\phi\ge 0$ in $I\times\partial\Omega$.
\end{hyp0}

\begin{thm}
\label{existence-bis}
Let Hypotheses $\ref{ipofisi}$ be satisfied and fix $s\in I$. Then, for any $f\in C_b(\Omega)$, the problem \eqref{NAPC} admits a unique
bounded classical solution $u$. The function $u$ belongs to $C^{1+\alpha/2,2+\alpha}_{\rm loc}(\overline{\Omega}_{(s,+\infty)})$ and
\begin{equation}
\|u(t,\cdot)\|_{\infty}\le Me^{H(t-s)}\|f\|_{\infty},\qquad\;\,t>s,
\label{giorgino}
\end{equation}
where $M=(\inf_{\Omega}\phi)^{-1}\|\phi\|_{\infty}$.
Finally, for any $f\in C^{2+\alpha}_c(\Omega)$, the unique bounded classical solution to problem
\eqref{NAPC} belongs to $C^{1+\alpha/2,2+\alpha}_{\rm loc}(\overline{\Omega_{(s,+\infty)}})$ and, if $f\ge 0$ does not identically vanish,
then $u>0$ in $\Om_{(s,+\infty)}$.
\end{thm}

\begin{proof}
Let us fix a function $f\in C_b(\Omega)$. We point out that $u$ is a bounded classical solution to problem \eqref{NAPC} if and only if the function
$v:\overline{\Omega_{(s,+\infty)}}\to\R$, defined by
$v(t,x):=e^{-H(t-s)}(\phi(x))^{-1}u(t,x)$ for any $(t,x)\in \overline{\Omega_{(s,+\infty)}}$, is a bounded classical solution to the Cauchy problem
\begin{equation}
\left\{
\begin{array}{ll}
D_tv(t,x)=(\tilde{\mathcal A} v)(t,x), & (t,x)\in (s,+\infty)\times\Omega,\\[1mm]
(\tilde {\mathcal B}v)(t,x)= 0, & (t,x)\in (s,+\infty)\times\partial\Omega,\\[1mm]
v(s,x)=(\phi(x))^{-1}f(x), &x\in\Omega,
\end{array}
\right.
\label{pb-per-v}
\end{equation}
where
\begin{align*}
\tilde{\mathcal A}v=({\mathcal A}+c)v+\frac{2}{\phi}\langle Q\nabla\phi,\nabla_xv\rangle
-\left (H-\frac{{\mathcal A}\phi}{\phi}\right )v,\qquad\;\, \tilde {\mathcal B}v=\langle \beta,\nabla_xv\rangle+\frac{{\mathcal B}\phi}{\phi}v.
\end{align*}

Clearly, the coefficients of operators $\tilde {\mathcal A}$ and $\tilde {\mathcal B}$ satisfy Hypotheses \ref{hyp1} (note that
the potential of the operator $\tilde\A$ is nonnegative in $\Omega_I$).
Moreover, Hypotheses \ref{hyp1-bc} are satisfied as well and, by Hypotheses \ref{ipofisi}, it follows that
 $({\mathcal B}\phi)/\phi\ge 0$ in $(s,+\infty)\times\partial\Omega$.
Finally, we note that, for any bounded interval $J\subset I$, the function $\overline\varphi_J$, defined by
$\overline\varphi_J(t,x)=e^{-H(t-s)}(\phi(x))^{-1}\varphi_J(t,x)$ for any
$(t,x)\in \overline{\Om_J}$, satisfies Hypothesis \ref{hyp-3} with the same $\lambda$ and the operators
${\mathcal A}$ and ${\mathcal B}$ being replaced by
$\tilde{\mathcal A}$ and $\tilde{\mathcal B}$.
We can thus apply the results in Subsection \ref{subsect-3.1} and deduce that the problem \eqref{pb-per-v} admits a unique bounded classical
solution $v$, which in addition belongs to $C^{1+\alpha/2,2+\alpha}_{\rm loc}(\overline{\Omega}_{(s,+\infty)})$ and satisfies the estimate
$\|v(t,\cdot)\|_{\infty}\le e^{H(t-s)}\|f/\phi\|_{\infty}$ for any $t \ge s$.
As a byproduct we deduce that problem \eqref{NAPC} admits a unique bounded classical solution which, in addition, belongs to
$C^{1+\alpha/2,2+\alpha}_{\rm loc}(\overline{\Omega}_{(s,+\infty)})$ and satisfies the inequality \eqref{giorgino}.

The last assertions follow from Proposition \ref{smoothdatum} and Theorem \ref{existence}, observing that the operator $f\mapsto f/\phi$
preserves positivity and it is an isomorphism from $C^{2+\alpha}_c(\Omega)$ into $C^{2+\alpha}_c(\Omega)$ and from
$C^{1+\alpha/2,2+\alpha}_{\rm loc}(\overline{\Omega_{(s,+\infty)}})$
into itself.
\end{proof}

\begin{rmk}
{\rm From the proof of Theorem \ref{existence-bis} and Remark \ref{chiusura} it follows that,
if the coefficients of the operator ${\mathcal B}$ are independent of $t$ and belong to
$C^{1+\alpha}_b(\overline\Omega)$ as well as the function $({\mathcal B}\phi)/\phi$, then the bounded classical solution $u_f$
to problem \eqref{NAPC} is continuous in $\overline{\Omega_{(s,+\infty)}}$ for any
$f\in C_b(\overline\Omega)$.
}
\end{rmk}

\section{The evolution operator: continuity properties and compactness}
\label{sect-evol}
Set $\Lambda=\{(t,s)\in I\times I:\,\,t>s\}$. In view of Theorems \ref{existence} and \ref{existence-bis}, the
family of bounded linear operators $\{G_{\mathcal B}(t,s): (t,s)\in \overline\Lambda\}$, defined in $C_b(\Omega)$ by $G_{\mathcal B}(t,t)={id}_{C_b(\Om)}$ and
$G_{\mathcal B}(t,s)f:= u_f(t,\cdot)$ for $t>s$, where $u_f$ is the unique solution to the problem \eqref{NAPC} with $f \in C_b(\Omega)$,
gives rise to an evolution operator.
Each operator $G_{\mathcal B}(t,s)$ is positive. Moreover, estimates \eqref{estinf} and \eqref{giorgino} imply that
\begin{align}
&\|G_{\mathcal B}(t,s)f\|_{\infty}
\le
\left\{
\begin{array}{ll}
e^{-c_0(t-s)}\|f\|_{\infty}, & \gamma\ge 0,\\[2mm]
Me^{H(t-s)}\|f\|_{\infty}, & {\rm otherwise},
\end{array}
\right.
\qquad\;\,t>s,\;\,f\in C_b(\Omega).
\label{stima-oper-B}
\end{align}
In particular, if $\gamma\ge 0$, $G_{\mathcal B}(t,s)$ is a contraction.
On the other hand, for a general $\gamma$, the proof of Theorem  \ref{existence-bis} shows that
\begin{equation}\label{link}
G_{\mathcal B}(t,s)f=\phi e^{H(t-s)}\tilde G_{\tilde{\mathcal B}}(t,s)\left (\textstyle{\frac{f}{\phi}}\right ),\qquad\;\, t>s\in I,
\;\,x\in\Omega,\;\,f\in C_b(\Omega),
\end{equation}
where $\tilde G_{\tilde{\mathcal B}}(t,s)$ is the evolution operator associated to problem \eqref{pb-per-v}.
In what follows,
to simplify the notation we denote the evolution operator $\{G_{\mathcal B}(t,s): (t,s)\in \overline\Lambda\}$ simply by
$G_{\mathcal B}(t,s)$.

We now prove some continuity property of the evolution operator $G_{\mathcal B}(t,s)$,
which will be used in Section \ref{sect_graest}.

\begin{prop}\label{continuity}
Let $(f_n)\subset C_b(\Om)$ be a bounded sequence with respect to the sup-norm and let
$f\in C_b(\Omega)$.
\begin{enumerate}[\rm (i)]
\item
If $f_n$ converges pointwise to $f$ in $\Omega$, then, for any pair of compact sets $J\subset (s,+\infty)$ and $K\subset\Om$,
$G_{\mathcal B}(\cdot,s)f_n$ converges
to $G_{\mathcal B}(\cdot,s)f$ in $C^{1,2}(J\times K)$ as $n\to +\infty$.
\item
If $f_n$ converges locally uniformly to $f$ in $\Omega$,
then $G_{\mathcal B}(\cdot,s)f_n$ converges uniformly to
$G_{\mathcal B}(\cdot,s)f$ in $[s,T]\times K$ for any $T>s$ and any compact set $K\subset \Om$.
\end{enumerate}
\end{prop}

\begin{proof}
(i) Since $\sup_{n\in\N}\|f_n\|_{\infty}<+\infty$, estimate \eqref{stima-oper-B} and the interior Schauder estimates show that the sequence
$(G_{\mathcal B}(\cdot,s)f_n)$ is bounded in $C^{1+\alpha/2,2+\alpha}(J\times K)$ for any $J$ and $K$ as in the statement.
The compactness argument already used in the proof of Theorem \ref{existence} shows that, there exists a subsequence
$(G_{\mathcal B}(\cdot,s)f_{n_k})$ which converges in $C^{1,2}(J\times K)$ to some function $v$.
To infer that $v=G_{\mathcal B}(\cdot,s)f$, we observe that, for any $(t,s,x)\in\Lambda\times\Omega$ the map
$f\mapsto (G_{\mathcal B}(t,s)f)(x)$ defines a positive and bounded operator from $C_0(\Omega)$ to $\R$. The Riesz representation theorem
shows that there exists a family of positive and finite Borel measures $\{g(t,s,x,dy): (t,s,x)\in\Lambda\times\Omega\}$ such that
\begin{equation}
(G_{\mathcal B}(t,s)\psi)(x)=\int_{\Omega}\psi(y)g(t,s,x,dy),\qquad\;\,I\ni s<t,\;\,x\in\Omega,\;\,\psi\in C_0(\Omega).
\label{baglioni-0}
\end{equation}
From the proof of Theorem \ref{existence} we know that, if $(\psi_n)\subset C^{2+\alpha}_c(\Omega)$ is bounded with respect to the sup-norm and
converges to some $\psi\in C_b(\Omega)$, locally uniformly in $\Omega$, then $G_{\mathcal B}(t,s)\psi_n$ converges to $G_{\mathcal B}(t,s)\psi$
locally uniformly in $\Omega$ as well. Hence, \eqref{baglioni-0} can be extended to any $\psi\in C_b(\Omega)$. Writing it with $\psi$ being
replaced by $f_n$ and letting $n\to +\infty$, by dominated convergence we conclude that $v=G_{\mathcal B}(\cdot,s)f$. Since the limit is independent
of the sequence $(n_k)$, the whole sequence $(G_{\mathcal B}(\cdot,s)f_n)$ converges to  $G_{\mathcal B}(\cdot,s)f$ in
$C^{1,2}(J\times K)$.

(ii) In view of formula \eqref{link}, we can limit ourselves to dealing with the case when $\gamma\ge 0$. For this purpose, fix
$T>s$, a compact set $K\subset\Omega$ and $\varepsilon>0$. Further, let $\eta \in C^{\infty}_c(\Om)$ be such that
$\eta\equiv 1$ in $K$. We can split $f_n= \eta f_n+(1-\eta)f_n$ for any $n\in\N$. Then, by linearity
$G_{\mathcal B}(\cdot,s)f_n= G_{\mathcal B}(\cdot,s)(\eta f_n)+G_{\mathcal B}(\cdot,s)((1-\eta)f_n)$. Since $\eta f_n \in C_c(\Om)$
and converges uniformly to $\eta f$, by \eqref{stima-oper-B} we deduce that $G_{\mathcal B}(\cdot,s)(\eta f_n)$ converges to
$G_{\mathcal B}(\cdot,s)(\eta f)$ uniformly in $\Om_{[s,T]}$. Hence, we just need to prove
that $G_{\mathcal B}(\cdot,s)((1-\eta)f_n)$ converges to $G_{\mathcal B}(\cdot,s)((1-\eta)f_n)$ locally uniformly in $[s,T]\times K$.
The arguments in the proof of Theorem \ref{existence} show that
\begin{eqnarray*}
|G_{\mathcal B}(t,s)((1-\eta)f_n)(x)|\le \sup_{n\in\N}\|f_n\|_{\infty}[1\!-\!(G_{\mathcal B}(t,s)\eta)(x)],\qquad\,t>s,\;x\in\overline\Omega,\;n\in\N.
\end{eqnarray*}
Therefore, as $t\to s^+$, $G_{\mathcal B}(t,s)((1-\eta)f_n)$ and $G_{\mathcal B}(t,s)((1-\eta)f)$ converge to $0$ uniformly in $K$
and uniformly with respect to $n\in\N$. Hence, we can determine $\delta\in (0,T-s)$ such that
\begin{equation}
\|G_{\mathcal B}(\cdot,s)((1-\eta)f_n)\|_{C([s,s+\delta]\times K)}+\|G_{\mathcal B}(\cdot,s)((1-\eta)f)\|_{C([s,s+\delta]\times K)}\le
\frac{1}{2}\varepsilon.
\label{AA}
\end{equation}
On the other hand, property (i) implies that $G_{\mathcal B}(\cdot,s)((1-\eta)f_n)$ converges to $G_{\mathcal B}(\cdot,s)((1-\eta)f)$
uniformly in $[s+\delta,T]\times K$. Thus, there exists $n_0\in\N$ such that
$\|G_{\mathcal B}(\cdot,s)((1-\eta)f_n)-G_{\mathcal B}(\cdot,s)((1-\eta)f)\|_{C([s+\delta,T]\times K)}\le\varepsilon/2$ for any $n\ge n_0$.
From this estimate and \eqref{AA} we conclude that $G_{\mathcal B}(\cdot,s)((1-\eta)f_n)$ tends to $G_{\mathcal B}(\cdot,s)((1-\eta)f_n)$
uniformly in $[s,T]\times K$.
\end{proof}

\begin{rmk}
\label{rmk-strong-feller}
{\rm The proof of Proposition \ref{continuity} shows that
\begin{equation}
(G_{\mathcal B}(t,s)f)=\int_{\Omega}f(y)g(t,s,x,dy),\qquad\;\,t>s\in I,\;\,x\in\Omega
\label{repres-G}
\end{equation}
for any $f\in C_b(\Omega)$. Since any Borel bounded function $f:\Omega\to\R$ can be approximated
pointwise by a bounded sequence in $C_b(\Omega)$, the dominated convergence theorem allows
to extend the evolution operator $G_{\mathcal B}(t,s)$ to all the bounded and Borel measurable
functions, via formula \eqref{repres-G}. Moreover, the same arguments in the proof of Proposition
\ref{continuity}(i) show that $G_{\mathcal B}(\cdot,s)f\in C^{1+\alpha/2,2+\alpha}_{\rm loc}(\Omega_{(s,+\infty)})$
for any bounded and Borel measurable function $f$. In particular, this shows that $G(t,s)$ is Strong Feller.}
\end{rmk}

In the following proposition, under an additional smoothness assumption on the coefficients of
the operator $\A$, we show that each measure $g(t,s,x,dy)$ is absolutely continuous with respect the Lebesgue measure,
and we prove some smoothness properties of its density.

\begin{prop}
Assume that $q_{ij}\in L^{\infty}_{\rm loc}(I;W^{1,p}_{\rm loc}(\Om))$ for some $p>d+2$.
Then, there exists a unique Green function $g_{\mathcal{B}}:\Lambda\times \Om\times \Om\to (0,+\infty)$ associated with
 the Cauchy problem \eqref{NAPC}, i.e.,
\begin{equation}
\label{green}
(G_{\mathcal{B}}(t,s)f)(x)= \int_\Om g_{\mathcal{B}}(t,s,x,y)f(y)\,dy,\qquad\;\, s<t,\;\, x\in \Om,
\end{equation}
for every $f \in C_b(\Omega)$. Moreover, for any $s \in I$ and $y \in \Om$, the function $g_{\mathcal B}(\cdot,s,\cdot,y)$ belongs to
$C^{1+\alpha/2,2+\alpha}_{{\rm {loc}}}(\Om_{(s,+\infty)})$ and satisfies $D_t g_{\mathcal B}-\mathcal{A}g_{\mathcal B}\equiv 0$ in
$\Omega_{(s,+\infty)}$. Finally, $g_{\mathcal B}(t,s,x,\cdot)\in L^1(\Om)$ for every $(t,s)\in \Lambda,\, x\in \Om$ and
\begin{equation}
\label{d_conv}
\| g_{\mathcal B}(t,s,x,\cdot)\|_{L^1(\Om)}\leq
\left\{
\begin{array}{ll}
e^{-c_0(t-s)}, & \gamma\ge 0,\\[1mm]
Me^{H(t-s)}, & {\rm otherwise},
\end{array}
\right.
\qquad\;\, (t,s)\in \Lambda,\, x\in \Om,
\end{equation}
where $H$ and $M$ are as in \eqref{giorgino}.
\end{prop}

\begin{proof}
We split the proof into four steps. In the first three steps we consider the case when $\gamma \ge 0$.
In the last one, we address the general case.

{\em Step 1.} First, we prove that there exists a function $g_{\mathcal B}:\Lambda\times \Om\times \Om\to (0,+\infty)$ such that \eqref{green} holds.
For this purpose, we consider the evolution operator $G_n(t,s)$ associated to the Cauchy problem \eqref{p_n} in $C_b(\Om)$.
It is well known (see...) that for every $n \in \N$ there exists a function $g_n: \Lambda\times \Om\times\Om \to \R^+$ such that,
for any $I\ni s<t$ and $x\in\Omega$, the function $g_n(t,s,x,\cdot)$ belongs to $L^1(\Omega)$ and
\begin{equation}
\label{repr-Gn}
(G_n(t,s)f)(x)=\int_{\Omega}f(y)g_n(t,s,x,y)dy, \qquad\;\, f \in C_b(\Om),\;\, n \in \N.
\end{equation}

Moreover, the function $g_n(\cdot,s,\cdot,y)$ belongs to $C^{1+\alpha/2,2+\alpha}_{\rm loc}(\Omega_{(s,+\infty)})$ for any $s\in I$, $y \in \Om$
and $D_t g_n(\cdot,s,\cdot,y)= \mathcal{A}^{(n)}g_n(\cdot,s,\cdot,y)$ in $\Omega_{(s,+\infty)}$. Finally, the maximum
principle in Proposition \ref{nsMP} immediately implies that $\|g_n(t,s,x,\cdot)\|_{L^1(\Omega)}\le 1$
for any $t>s\in I$ and $x\in\Omega$.

We now fix $t\in I$, $s_1,s_2\in (-\infty,t)\cap I$ and $x\in\Omega$. Formula \eqref{bietola_appen} shows that $D_s((G_n(t,s)\psi(s,\cdot))(x))=(G_n(t,s)D_s\psi(s,\cdot))(x)-
(G_n(t,s)\A^{(n)}(s)\psi(s,\cdot))(x)$, for any $(s,x)\in\Omega_{I\cap (-\infty,t]}$ and any $\psi\in C^{1,2}_c(\Omega_{(s_1,s_2)})$. Hence,
taking \eqref{repr-Gn} into account we deduce that
\begin{align}\label{scegli}
\int_{\Omega_{(s_1,s_2)}}(D_r\psi(r,y)-{\mathcal A}^{(n)}(r)\psi(r,y))g_n(t,r,x,y)dr\, dy=0,
\end{align}
since $\psi(s_1,\cdot)=\psi(s_2,\cdot)\equiv 0$.
Applying \cite[Cor. 3.9]{BogKryRoc01} to the measures $\mu_n(dr,dy)=g_n(t,r,x,y)dy dr$,  we deduce that $g_n(t,\cdot,x,\cdot)$ is
locally $\theta$-H\"older continuous in $\Omega_{(s_1,s_2)}$ for some $\theta\in (0,1)$, and, by the arbitrariness of $s_1<s_2<t$,
$g_n(t,\cdot,x,\cdot)\in C^{\theta}_{\rm loc}(\Omega_{(-\infty,t)\cap I})$ for any $(t,x)\in \Omega_I$. Moreover, an inspection
of the proof of \cite[Thm 3.8]{BogKryRoc01}, shows that the $C^{\theta}$-norm of $g_n(t,\cdot,x,\cdot)$ over any compact set
$[a,b]\times K\subset\Omega_{(-\infty,t)\cap I}$ can be bounded from above by a constant, independent of $n$.
Hence, a straightforward compactness argument allows to prove that, for any fixed $(t,x)\in\Omega_{(s,+\infty)}$,
there exist a subsequence $(n_k)\subset \N$ and a function
$g_{t,x}\in C^{\theta}_{\rm loc}(\Om_{(-\infty,t)\cap I})$ such that $g_n(t,\cdot,x,\cdot)$ converges to $g_{t,x}$ locally uniformly
in $\Om_{(-\infty,t)\cap I}$. Moreover, since $\|g_n(t,s,x,\cdot)\|_{L^1(\Omega)}\le 1$ for any $n\in\N$,
Fatou lemma shows that $g_{t,x}(s,\cdot)\in L^1(\Omega)$ for any $I\ni s<t$. Hence, we can write \eqref{repr-Gn}, with $n$ being replaced by
$n_k$, and let $k\to +\infty$, using the dominated convergence theorem and taking Theorem \ref{existence} into account, to get
\begin{equation}
(G_{\mathcal B}(t,s)f)(x)=\int_{\Omega}f(y)g_{t,x}(s,y)\,dy, \qquad\;\, f \in C_b(\Om).
\label{fird}
\end{equation}
Formula \eqref{fird} shows also that the function $g_{t,x}$ does not depend on subsequence $(n_k)$
and, therefore all the sequence $(g_n(t,\cdot,x,\cdot))$ converge to $g_{t,x}$ locally uniformly in $\Omega_{(-\infty,t)\cap I}$.
Formula \eqref{green} follows with $g_{\mathcal B}(t,s,x,y)=g_{t,x}(s,y)$.
Moreover, estimate \eqref{stima-oper-B} and formula \eqref{fird} lead to \eqref{d_conv}.

{\em Step 2.} Let us now prove that $g_{\mathcal B}(\cdot,s,\cdot, y)\in C^{1+\alpha/2, 2+\alpha}_{\rm loc}(\Om_{(s,+\infty)})$ and solves
$D_t g_{\mathcal B}(\cdot,s,\cdot,y)-\mathcal{A}g_{\mathcal B}(\cdot,s,\cdot,y)=0$ in $\Om_{(s,+\infty)}$. For this purpose, we begin by observing that,
since $D_tg_n(\cdot,s,\cdot,y)-\A^{(n)} g_n(\cdot,s,\cdot,y)=0$ in $\Omega_{(s,+\infty)}$, the classical
interior Schauder estimates yield that
\begin{equation}\label{schauder}
\|g_n(\cdot, s,\cdot,y)\|_{C^{1+\alpha/2,2+\alpha}(\Om'_{(T_1,T_2)})}\le c \|g_n(\cdot, s,\cdot,y)\|_{L^\infty(\Om''_{(T_0,T_2)})},
\end{equation}
for any $s<T_0<T_1<T_2$, any $\Om'\Subset \Om''\Subset \Om$ and some positive constant $c$ independent of $n$,
since the coefficients of the operator $\A^{(n)}$ converge to the coefficients of the operator $\A$, locally uniformly
in $\Omega_{(s,+\infty)}$.

We claim that, for every $s\in I$ and $y \in \Om$, the function $g_n(\cdot,s,\cdot,y)$ is bounded in $\Om''_{(T_0,T_2)}$, uniformly
with respect to $n$. So, let us fix $s\in I$, $y \in \Om$ and denote by $(t_h)$ and $(x_k)$ two countable sets dense in $[T_0,T_2+1]$
and in $\Om''$, respectively.
By Step 1, $(g_n(t_h,s,x_k,\cdot))$
converges locally uniformly in $\Om$ to $g_{\mathcal B}(t_h,s,x_k,\cdot)$ for any $h,k \in \N$, as $n\to +\infty$.
In particular, there exists a positive constant $c(t_h,x_k)$ such that $g_n(t_h,s,x_k,y)\le c(t_h,x_k)$ for every $h,k,n \in \N$.

Let $R>1$ be such that $s<T_0-2/R$ and let $2r:=\rm{dist}(\Om'', \partial \Om)$. Since
$D_tg_n(\cdot,s,\cdot,y)-{\mathcal A}^{(n)}g_n(\cdot,s,\cdot,y)=0$ in $[T_0-1/R,T_2+2]\times \bigcup_{x \in \Om''}\overline{B_r(x)}$,
using the arguments in \cite[Cor. 3.11]{BogKryRoc01} based on the Harnack inequality in \cite[Thm. 3]{AroSer68Loc},
we see that, if $\rho^2< \min\{r^2,R^{-1}\}$,
then there exists a positive constant $M_0$, independent of $n$, such that
$\max_{\overline{W_1}}g_n(\cdot,s,\cdot,y)\leq M_0 \min_{\overline{W_2}}g_n(\cdot,s,\cdot,y)$,
where $W_1=(t_h-\frac{3}{4}\rho^2, t_h-\frac{1}{2}\rho^2)\times B_{\rho/2}(x_k)$ and $W_2=(t_h-\frac{1}{4}\rho^2, t_h)\times B_{\rho/2}(x_k)$. Consequently,
$g_n(t,s,x,y)\leq M_0 g_n(t_h,s,x_k,y)\le M_0c(t_h,x_k)$
for every $t \in [t_h-\frac{3}{4}\rho^2, t_h-\frac{1}{2}\rho^2]$, $x \in \overline{B_{\rho/2}(x_k)}$.
Since $\Om''_{[T_0,T_2]}$ can be covered by a finite number of cylinders $[t_h-\frac{3}{4}\rho^2, t_h-\frac{1}{2}\rho^2]\times B_{\rho/2}(x_k)$,
from the previous chain of inequalities we deduce that $g_n(\cdot,s,\cdot,y)$ is uniformly bounded in
$\Om''_{[T_0,T_2]}$ by a constant independent of $n$, as it has been claimed.

From \eqref{schauder}, we now deduce that for any $s,y$ as above, the $C^{1+\alpha/2,2+\alpha}$-norm of the function
$g_n(\cdot,s,\cdot,y)$ over the cylinder $\Omega'_{(T_1,T_2)}$ can be bounded from above by a constant independent of $n$.
Consequently, by the Arzel\`a-Ascoli theorem, there exist an increasing sequence $(n_k)\subset\N$ and a function
$\tilde{g}_{s,y} \in C^{1+\alpha/2,2+\alpha}(\Omega'_{(T_1,T_2)})$ such that $g_{n_k}(\cdot,s,\cdot,y)$ converges to
$\tilde{g}_{s,y}$ in $C^{1,2}(\Omega'_{(T_1,T_2)})$, as $k \to +\infty$. Clearly $\tilde{g}_{s,y}$ belongs
to $C^{1+\alpha/2,2+\alpha}(\Omega'_{(T_1,T_2)})$ and satisfies the differential equation
in \eqref{NAPC} in $\Omega'_{(T_1,T_2)}$. Since $g_n$ converges to $g_{\mathcal B}$ pointwise in $\Lambda\times\Omega\times\Omega$,
we can infer that $\tilde g_{s,y}=g_{\mathcal B}(\cdot,s,\cdot,y)$. The arbitrariness of $T_1$, $T_2$, $\Omega'$ allows us to conclude
that $g_{\mathcal B}(\cdot,s,\cdot,y)\in C^{1+\alpha/2,2+\alpha}_{\rm loc}(\Omega_{(s,+\infty)})$ and
$D_tg_{\mathcal B}(\cdot,s,\cdot,y)-\A g_{\mathcal B}(\cdot,s,\cdot,y)=0$ in $\Omega_{(s,+\infty)}$.

{\em Step 3.} Here, we prove that $g_{\mathcal B}$ is strictly positive in $\Lambda\times\Om\times\Om$.
From Theorem \ref{existence} we know that $G_{\mathcal B}(t,s)f$ is strictly positive in $\Omega$, for any $f\in C_b(\Omega)$. This
implies that $g_{\mathcal B}$ is nonnegative. Indeed, if $g(t_0,s_0,x_0,y_0)<0$ for some $(t_0,s_0,x_0,y_0)\in\Lambda\times\Omega\times\Omega$,
since the function $g_{\mathcal B}(t_0,s_0,x_0,\cdot)$ is continuous in $\Omega$, there would exist $\sigma>0$ such that $g(t_0,s_0,x_0,\cdot)<0$
in $B_{\sigma}(y_0)$. Taking a function $f\in C_b(\Omega)$ such that $\chi_{B_{\sigma/2}(y_0)}\le f\le\chi_{B_{\sigma}(y_0)}$ we would get $(G(t_0,s_0)f)(x_0)<0$, which is a contradiction.

Suppose that $g_{\mathcal B}(t_0,s_0,x_0,y_0)=0$ at some point $(t_0,s_0,x_0,y_0)$. Fix $t_1<s_1<s_2<s_0<t_2<t_0$ and $R>0$
such that $B_R(y_0)\Subset \Om$. Letting $n \to +\infty$ in \eqref{scegli} we get that $(D_r-\mathcal{A})^*g_{\mathcal B}(t_0,\cdot,x_0,\cdot)=0$ in $(t_1,t_2)\times B_R(y_0)$. Hence, using again the arguments in \cite[Cor. 3.11]{BogKryRoc01} we deduce that $g_{\mathcal B}(t_0, \cdot, x_0, \cdot)=0$ in $(s_1,s_2)\times B_R(y_0)$.
Now, fix $s \in(s_1,s_2)$ and a nonnegative function $f \in C_c(B_R(y_0))$ such that $f(y_0)>0$. From \eqref{green} it follows that
$(G_{\mathcal B}(t_0,s)f)(x_0)=0$, which cannot be the case since $G_{\mathcal B}(t_0,s)f>0$ in $\Omega$.

{\em Step 4.}
Finally, in the general case when no assumption on the sign of $\gamma$ is assumed, we can apply the first part of the proof to
the evolution operator $\tilde{G}_{\tilde{\mathcal B}}(t,s)$ associated with  the Cauchy problem \eqref{pb-per-v} in $C_b(\Om)$.
Then, there exists a unique function $\tilde g_{\tilde{\mathcal{B}}}:\Lambda\times \Om\times \Om\to (0,+\infty)$
such that
\begin{equation*}
(\tilde{G}_{\tilde{\mathcal B}}(t,s) f)(x)= \int_\Om \tilde g_{\tilde{\mathcal B}}(t,s,x,y)f(y)\,dy,\qquad\;\, s<t,\;\, x\in \Om,\;\, f \in C_b(\Om).
\end{equation*}
Moreover, for any $s \in I$ and any $y \in \Om$,
$\tilde g_{\tilde{\mathcal B}}(\cdot,s,\cdot,y)\in C^{1+\alpha/2,2+\alpha}_{{\rm {loc}}}(\Omega_{(s,+\infty)})$,
satisfies $D_t\tilde g_{\tilde{\mathcal B}}-\tilde{\mathcal{A}}\tilde g_{\tilde{\mathcal B}}=0$ in $\Omega_{(s,+\infty)}$ and
$\|\tilde g_{\tilde{\mathcal B}}(t,s,x,\cdot)\|_{L^1(\Om)}\le 1$ for any $(t,s)\in \Lambda$, $x \in \Om$. Now,
taking into account formula \eqref{link} we conclude that the function $g_{\mathcal B}$, defined by
$g_{\mathcal B}(t,s,x,y)= \phi(x)(\phi(y))^{-1}e^{H(t-s)}\tilde g_{\tilde{\mathcal B}}(t,s,x,y)$ for any
$(t,s,x,y)\in\Lambda\times\Omega\times\Omega$, satisfies the claim and estimate \eqref{d_conv}
holds.
\end{proof}

\subsection{Compactness}

We now provide a sufficient condition for $G_{\mathcal B}(t,s)$ to be compact in $C_b(\Omega)$.
In the case when $\Omega$ is replaced by $\Rd$, the compactness of the evolution operator $G_{\mathcal B}(t,s)$
has been studied in \cite{AngLor10Com,Lun10Com}.

\begin{rmk}
\label{rmk-Gn-to-GB}
{\rm
Suppose that $\gamma\ge 0$ and let $G_n(t,s)$ be the evolution operator associated with the pair $(\A^{(n)},{\mathcal B}^{(n)})$ in
$C_b(\Omega)$, introduced in the proof of Proposition \ref{smoothdatum}. As it has been shown, for any $f\in C^{2+\alpha}_c(\Omega)$,
$G_n(t,s)f$ converges to $G_{\mathcal B}(t,s)f$ locally uniformly in
$\Omega$, as $n\to +\infty$, for any $I\ni s<t$. Actually, this happens also if $f\in C_b(\Omega)$.
If $f\in C_c(\Omega)$, let the sequence $(f_k)\subset C^{2+\alpha}_c(\Omega)$
converge to $f$ uniformly in $\Omega$. Then,
splitting $G_n(t,s)f-G_{\mathcal B}(t,s)f=(G_n(t,s)f-G_n(t,s)f_k)+(G_n(t,s)f_k-G_{\mathcal B}(t,s)f_k)
+(G_{\mathcal B}(t,s)f_k-G_{\mathcal B}(t,s)f)$ and observing that, by the maximum principle in Proposition \ref{nsMP}, $\|G_n(t,s)f\|_\infty \le \|f\|_\infty$ for any $I \ni s<t$ and $f \in C_b(\Om)$, we easily deduce that
\begin{align*}
\|G_n(t,s)f-G_{\mathcal B}(t,s)f\|_{C(K)}
\le 2\|f_k-f\|_\infty+\|G_n(t,s)f_k-G_{\mathcal B}(t,s)f_k\|_{C(K)},
\end{align*}
for any compact set $K\subset\Omega$.
Letting, first $n$ and then $k$ tend to $+\infty$, we deduce that $G_n(t,s)f$ tends to $G_{\mathcal B}(t,s)f$ locally uniformly in $\Omega$.

In the general case, we can argue as in the proof of
Theorem \ref{existence}, replacing $u_{f_n}$ with $G_n(\cdot,s)f$, $u_{(1-\psi)f_n}$ with $G_n(\cdot,s)((1-\psi)f)$
and $u_{\psi f_n}$ with $G_n(\cdot,s)(\psi f)$. Splitting $G_n(t,s)f=G_n(t,s)(\psi f)+G_n(t,s)((1-\psi)f)$ and observing that the estimate
$|G_n(t,s)((1-\psi)f)|\le \|f\|_{\infty}(1-G_n(t,s)\psi)$ holds for any $t>s$, we get
$|(G_n(t,s)f)(x)-f(x)|\le |(G_n(t,s)(f\psi))(x)-f(x)\psi(x)|+\|f\|_{\infty}(1-(G_n(t,s)\psi)(x))$ for any $(t,x)\in (s,+\infty)\times K$.
Let $(G_{n_k}(\cdot,s)f)$ be a subsequence which converges, locally uniformly in $\Omega_{(s,+\infty)}$ to
a function $u$, which turns out to solve the differential equation $D_tu-\A u=0$ in $\Omega_{(s,+\infty)}$ and satisfies the boundary condition
$\B u=0$ on $(s,+\infty)\times\partial\Omega$. Since $G_n(\cdot,s)(f\psi)$ converges to $G_{\mathcal B}(\cdot,s)(f\psi)$ locally uniformly
in $\Omega_{(s,+\infty)}$, we can write the previous estimate with $n$ being replaced by $n_k$ and let $k\to +\infty$,
to infer that $|u(t,x)-f(x)|\le |(G_{\mathcal B}(t,s)(f\psi))(x)-f(x)\psi(x)|+\|f\|_{\infty}(1-(G_{\mathcal B}(t,s)\psi)(x))$ and conclude that
$u$ is continuous on $\{s\}\times K$. Hence, $u=G_{\mathcal B}(\cdot,s)f$. Since the limit is independent of the sequence $(n_k)$, the whole sequence $G_n(\cdot,s)f$ converges to
$G_{\mathcal B}(\cdot,s)f$ locally uniformly in $\Omega$.}
\end{rmk}

\begin{thm}
\label{thm-comp} Assume that there exist a bounded interval $J\subset I$, $c_1,c_2>0$ a
positive function $\psi\in C^2(\Omega)$ and $\varepsilon>0$, such that $\lim_{x \in \Om, |x|\to \infty}\psi(x)=\infty$ and
$(\mathcal{A}(t)\psi)(x)\leq c_2-c_1(\psi(x))^{1+\varepsilon}$ for any $(t,x)\in J\times\Omega$.
Then, $G_{\mathcal B}(t,s)$ is compact in $C_b(\Omega)$ for any $(t,s)\in \Lambda\cap J^2$.
\end{thm}
\begin{proof}
In view of formula \eqref{link}, we can limit ourselves to considering the case when $\gamma\ge 0$. Moreover, it suffices to prove that
\begin{equation}\label{tight}
\lim_{n\to +\infty}\sup_{x\in\Omega}g_\B(t,s,x,\Omega\setminus\Omega_n)=0,\qquad\;\,(t,s)\in\Lambda\cap J^2,
\end{equation}
where the measures $g_\B(t,s,x,dy)$ are defined in \eqref{repres-G}.
Formula \eqref{tight} implies that for any $(t,s)\in\Lambda\cap J^2$, the evolution operator $G_{\mathcal B}(t,s)$ is the limit
of the sequence of operators $(S_n)$ defined by $S_nf=G_{\mathcal B}(t,r)(\chi_{\Omega_n}G_{\mathcal B}(r,s)f)$ for any $f\in C_b(\Omega)$ and
$n\in\mathbb N$, where $r=(s+t)/2$. Indeed, $\|G_{\mathcal B}(t,s)-S_n\|_{\mathcal{L}(C_b(\Omega))}\le\sup_{x\in\Omega}g_\B(t,r,x,\Omega\setminus \Omega_n)$.
Note that each operator $S_n$ is well defined, in view of Remark
\ref{rmk-strong-feller}, and is compact. Indeed, if $(f_k)$ is
a bounded sequence in $C_b(\Omega)$, then the interior Schauder estimates shows that the sequence $(G_{\mathcal B}(r,s)f_k)$ is bounded in
$C^{2+\alpha}(\Omega_n)$. Hence, it admits a subsequence $(G_{\mathcal B}(r,s)f_{k_m})$ uniformly converging in $\Omega_n$. As a byproduct, using formula \eqref{repres-G}, we deduce that
$S_nf_{k_m}$ converges uniformly in $\Omega$ as $m\to +\infty$, whence $S_n$ is a compact operator.

To prove \eqref{tight} we argue as follows: first we prove that $\psi$ is integrable with respect to
the measures $g(t,s,x,dy)$ for any $(t,s)\in \Lambda\cap J$ and $x\in\Omega$ (in what follows, with
a slight abuse we denote by $(G_{\mathcal B}(t,s)\psi)(x)$ the integral of the function $\psi$ with
respect to the measure $g(t,s,x,dy)$, whenever the integral makes sense, even if it is not finite). Then, we
show that, for any $s \in J$ and $\delta>0$, $(G_{\mathcal B}(t,s)\psi)(x)$ is bounded in $([s+\delta,+\infty)\cap J)\times\Omega$. This is enough
for our aims. Indeed,
\begin{align*}
g(t,s,x,\Omega\setminus\Omega_n)&\leq \frac{1}{k_n}\int_{\Omega\setminus\Omega_n}\psi(y)g(t,s,x,dy)
\leq \frac{1}{k_n}(G_{\mathcal B}(t,s)\psi)(x)
\leq \frac{M}{k_n},
\end{align*}
for some positive constant $M$, where $k_n:=\inf\{\psi(y): y\in\Omega\setminus\Omega_n\}$
tends to $+\infty$ as $n\to +\infty$.

We split the rest of the proof in three steps.

{\em Step 1.} For any $n\in\N$, let $\psi_n=\phi_n\circ\psi$, where $\phi_n \in C^{2}([0,+\infty))$ is
an increasing and concave function such that $\phi_n(r)= r$ for any $r \in (0,n)$ and $\phi_n(r)=n+1$ for any $r \in (n+2,+\infty)$.
Clearly, for any $n\in\N$, $\psi_n$ belongs to $C^2(\Omega)$  and is constant outside a compact set.
Let us prove that
$(G_{\mathcal B}(t,\sigma)\A(\sigma)\psi_n)(x)$ is well defined for any $(t,\sigma)\in \Lambda\cap J^2$, $x\in \Omega$ and
\begin{align}
(G_{\mathcal B}(t,r)\psi_n)(x)-(G_{\mathcal B}(t,s)\psi_n)(x)\ge &-\int_s^r(G_{\mathcal B}(t,\sigma)(\phi_n'(\psi)\A(\sigma)\psi))(x)
d\sigma,
\label{est-2}
\end{align}
for any $(t,r), (t,s)\in \Lambda \cap J^2$, $x\in\Omega$ and any $n\in\N$. In the rest of this step $s$, $r$, $t$, $x$ and $n$ are arbitrarily
fixed as above.

Let $G_k(t,s)$ be the evolution operator associated in $C_b(\Om)$ with  $(\mathcal{A}^{(k)},\mathcal{B}^{(k)})$, introduced in the proof Proposition \ref{smoothdatum}.
Applying Theorem \ref{appendicite}, with $f$ being replaced with $\zeta_n:=\psi_n-n-1$, observing that $\zeta_n-G_k(t,s)\zeta_n\le \psi_n-G_k(t,s)\psi_n$ and recalling that $G_k(t,\sigma)$ preserves positivity, we deduce that
\begin{align}
\psi_n(x)-(G_k(t,s)\psi_n)(x)\ge
&-\int_s^t(G_k(t,\sigma)((\A^{(k)}(\sigma)+c^{(k)}(\sigma,\cdot))\psi_n)(x)d\sigma\notag\\
&+\int_s^t(G_k(t,\sigma)(c^{(k)}(\sigma,\cdot)\psi_n\vartheta_m))(x)d\sigma,
\label{mare}
\end{align}
for any $k,m\in\N$ and $I\in s<t$, where $\vartheta_m\in C_c(\Omega)$ is supported in $\Omega_m$ and its image is
contained in $[0,1]$.  Here, we have taken into account that $c^{(k)(\sigma,\cdot)}\psi_n\ge c^{(k)(\sigma,\cdot)}\psi_n\vartheta_m$ and that $G_k(t,\sigma)$
preserves positivity. Writing \eqref{mare} with $t=r$ and applying Lemma \ref{lem-int}, with $T=G_k(t,r)$ (noting that
the functions $G_k(r,\cdot)((\A^{(k)}+c^{(k)}(\sigma,\cdot)\psi_n)$ and
$G_k(r,\cdot)(c^{(k)}(\sigma,\cdot)\psi_n\vartheta_m)$ are continuous in $\Omega_{I\cap (-\infty,r]}$, yields
\begin{align}
(G_k(t,r)(\psi_n)(x)-(G_k(t,s)\psi_n)(x)\ge
&-\int_s^r(G_k(t,\sigma)((\A^{(k)}(\sigma)+c^{(k)}(\sigma,\cdot))\psi_n)(x)d\sigma\notag\\
&+\int_s^r(G_k(t,\sigma)(c^{(k)}(\sigma,\cdot)\psi_n\vartheta_m))(x)d\sigma.
\label{mare-1}
\end{align}

We now want to let $k\to +\infty$ in the first and last side of \eqref{mare} and \eqref{mare-1}. For this purpose,
we observe that, for $k$ large enough (which is independent of $\sigma\in [s,r]$)
$(\A^{(k)}(\sigma)+c^{(k)}(\sigma,\cdot))\psi_n=(\A(\sigma)+c(\sigma))\psi_n\in C_b(\Omega)$
and $c^{(k)}\psi_n\vartheta_m=c\psi_n\vartheta_m\in C_b(\Omega)$.
Therefore, taking Remark \ref{rmk-Gn-to-GB} into account, by dominated convergence we conclude that
formula \eqref{mare-1} holds true with $G_k(\cdot,\cdot)$, $\A^{(k)}$ and $c^{(k)}$ being replaced, respectively, by
$G_{\mathcal B}(\cdot,\cdot)$, $\A$ and $c$.
Letting $m\to +\infty$ by monotone convergence, shows that
$(G_{\mathcal B}(t,\sigma)(c(\sigma,\cdot)\psi_n))(x))$ is finite for almost every
$\sigma\in (s,t)$, $(G_{\mathcal B}(t,\cdot)(c\psi_n\vartheta_m))(x)$ tends to $(G_{\mathcal B}(t,\cdot)(c\psi_n))(x)$
in $L^1((s,t))$ as $m\to +\infty$ and
\begin{align*}
(G_{\mathcal B}(t,r)\psi_n)(x)-(G_{\mathcal B}(t,s)\psi_n)(x)\ge &-\int_s^r(G_{\mathcal B}(t,\sigma)(\A(\sigma)\psi_n))(x)
d\sigma.
\end{align*}
Using the inequality $r\psi'_n(r)-\psi_n(r)\le 0$ for any $r\ge0$ (which easily follows recalling that $\phi_n$ is concave) we deduce that
$\A\psi_n\le \phi_n'(\psi)\A\psi$ and, hence, $G_{\mathcal B}(t,\cdot)\A\psi_n\le G_{\mathcal B}(t,\cdot)(\phi_n'(\psi)\A\psi)$,
and this latter function is integrable in $(s,r)$ since it differs from $G_{\mathcal B}(t,\cdot)\A\psi_n$ in bounded terms.
Estimate \eqref{est-2} now follows.

{\em Step 2.}
Here, using the results in Step 1, we prove that $(G_{\mathcal B}(t,\sigma)\A(\sigma)\psi)(x)$ is well defined for any
$\sigma\in J\cap (-\infty,t]$ and
\begin{equation}\label{est_1}
(G_{\mathcal B}(t,r)\psi)(x)-(G_{\mathcal B}(t,s)\psi)(x) \geq -\int_s^r (G_{\mathcal B}(t,\sigma)\mathcal{A}(\sigma)\psi)(x)d\sigma.
\end{equation}
We begin by observing that, by monotone convergence,
$(G_{\mathcal B}(t_2,t_1)\psi_n)(x)$ tends to $(G_{\mathcal B}(t_2,t_1)\psi)(x)$ as $n\to +\infty$, for any $(t_2,t_1)\in\Lambda$ and any $x\in\Omega$. This limit
might be, {\it apriori}, $+\infty$. We will show that this is not the case.
For this purpose, we use \eqref{est-2} with $r=t$, to infer that
\begin{align}
\psi_n(x)\ge &-\int_s^t(G_{\mathcal B}(t,\sigma)(\chi_{\{\A(\sigma)\psi>0\}}\phi'_n(\psi)\A(\sigma)\psi)(x)d\sigma\notag\\
&+\int_s^t(G_{\mathcal B}(t,\sigma)(\chi_{\{\A(\sigma)\psi<0\}}\phi'_n(\psi)|\A(\sigma)\psi|)(x)d\sigma.
\label{est-3}
\end{align}
Since the set $\{(\sigma,x)\in \Omega_J: (\A(\sigma)\psi)(x)>0\}$ is bounded (due to our assumption on $\psi$),
$G_{\mathcal B}(t,\sigma)(\chi_{\{\A(\sigma)\psi>0\}}\phi'_n(\psi)\A(\sigma)\psi)(x)$ converges in a dominated way to
$G_{\mathcal B}(t,\sigma)(\chi_{\{\A(\sigma)\psi>0\}}(\psi)\A(\sigma)\psi)(x)$ for any $\sigma\in [s,t]$.
Estimate \eqref{est-3} now show that the sequence
$((G_{\mathcal B}(t,\sigma)(\chi_{\{\A(\sigma)\psi<0\}}\phi'_n(\psi)|\A(\sigma)\psi|)(x))$ is
bounded in $L^1((s,t))$. Hence, by monotone convergence, it tends to
$(G_{\mathcal B}(t,\sigma)(\chi_{\{\A(\sigma)\psi<0\}}|\A(\sigma)\psi|)(x)$
in $L^1((s,t))$. Summing up, we have proved that
\begin{align*}
\lim_{n\to +\infty}\int_s^t(G_{\mathcal B}(t,\sigma)(\phi'_n(\psi)\A(\sigma)\psi)(x)d\sigma=
\int_s^t(G_{\mathcal B}(t,\sigma)(\A(\sigma)\psi)(x)d\sigma.
\end{align*}
Using again \eqref{est-2} (with $r=t$), we conclude that the sequence $(G_{\mathcal B}(t,s)\psi_n)(x))$ is bounded,
as claimed.

Now, taking the above results into account, we can let $n\to +\infty$ in \eqref{est-2} and obtain \eqref{est_1}.

{\em Step 3.}
Here, we prove that the function $(G_{\mathcal B}(t,\cdot)\psi)(x)$ is bounded, in any compact interval contained
in the $J\cap (-\infty,t)$, by a constant independent of $x$. Since $\A\psi\le c_2-c_1\psi^{1+\varepsilon}$ in $\Omega_J$
it follows that $G_{\mathcal B}(t,\cdot)\A\psi\le c_2G_{\mathcal B}(t,\cdot)\one-c_1G(t,\cdot)\psi^{1+\varepsilon}$.
In particular, this inequality shows that $(G(t,s)\psi^{1+\varepsilon})(x)<+\infty$.
H\"older inequality and the fact that $0<g(t,s,x,\Omega)=(G_{\mathcal B}(t,s)\one)(x)\leq 1$ for every $t>s$ and
$x\in\Omega$ show that
$\left ((G(t,s)\psi)(x)\right )^{1+\varepsilon}\leq (G(t,s)\psi^{1+\varepsilon})(x)$ for any $t>s\in I$ and $x\in\Omega$. Hence, from the above results and \eqref{est_1}, and recalling that $G_{\mathcal B}(t,s)\one\le\one$ for any $(t,s)\in\Lambda$, we get
\begin{equation}\label{est_5}
(G_{\mathcal B}(t,r)\psi)(x)-(G_{\mathcal B}(t,s)\psi)(x) \geq -c_2(r-s)
+c_1\int_s^r \left((G_{\mathcal B}(t,\sigma)\psi)(x)\right )^{1+\varepsilon}d\sigma.
\end{equation}

Let us set $\zeta(r):=(G_{\mathcal B}(t, t-r)\psi)(x)$ for any $r\in [0,\overline{r})$, where $\overline{r}=t-\inf I$.
Estimate \eqref{est_5} shows that the function $r\mapsto\zeta(r)-c_2r$ is decreasing. As a byproduct, $\zeta$ admits
left and right limits at any point $r\in (0,\overline{r})$. Moreover,
\begin{equation}
\lim_{r\to r_*^-}\zeta(r)\ge \zeta(r_*)\ge\lim_{r\to r_*^-}\zeta(r),\qquad\;\,r_*\in (0,\overline{r}).
\label{est-6}
\end{equation}

For any $x\in\Omega$, let $y(\cdot;x)$ denote the solution
of the differential equation $y'(r)=-c_1(y(r))^{1+\varepsilon}+c_2$, $r>0$, which satisfies the condition $y(0)=\psi(x)$.
Clearly, $y(\cdot;x)$ is defined in $[0,+\infty)$ and
\begin{eqnarray*}
{\mathcal H}(y(t)):=\int_{y(t;x)}^{+\infty}\frac{1}{c_1z^{1+\varepsilon}-c_2}dz\ge \delta,\qquad\;\,t\ge\delta.
\end{eqnarray*}
Hence $y(t;x)\in {\mathcal H}^{-1}((\delta,+\infty))$, which is a bounded set since $\lim_{\sigma\to +\infty}{\mathcal H}(\sigma)=0$.
Thus, the function $y(\cdot,x)$ is bounded by $M:={\mathcal H}^{-1}((\delta,\infty))$ for any $x\in\Omega$.

To conclude the proof, let us show that $\zeta\le y(\cdot;x)$ for any $x\in\Omega$.
We argue by contradiction: we suppose that there exist $s_0\in (0,\overline{r})$ and $x\in\Omega$ such that
$\zeta(s_0)>y(s_0;x)$, and we show that $\zeta>y(\cdot;x)$ in $[0,s_0]$ (this, of course, leads to a contradiction since
 $\zeta(0)=y(0;x)=\varphi(x)$). For this purpose, we begin by observing that
there exists $\delta_0>0$ such that $\zeta>y(\cdot;x)$ in $[s_0-\delta_0,s_0)$. Indeed, if this were not the case, there would exist a sequence $(s_n)$ converging to $s_0$ from the left such that $\zeta(s_n)\le y(s_n)$
for any $n\in\N$. Letting $n\to +\infty$ and taking \eqref{est-6} into account,
we would get to a contradiction.  Suppose that $\delta_0<s_0$. Then, there exists some $\overline s\in [0,s_0)$ such that $y(\overline s;x)\ge \zeta(\overline s)$ and
$y(\cdot;x)<\zeta$ in $(\overline s,s_0)$. As a consequence, $\int_{\overline s}^s|\zeta(\sigma)|^{1+\varepsilon}d\sigma
>\int_{\overline s}^s|y(\sigma;x)|^{1+\varepsilon}d\sigma$ for any $s\in (\overline s,s_0)$, and, using estimate \eqref{est_5} with $r_1=\overline s$, $r_2=s\in (\overline s,s_0)$, we  get
\begin{eqnarray*}
y(s;x)-\zeta(s)-(y(\overline s;x)-\zeta(s;x))\ge c_1\left (\int_{\overline s}^s|\zeta(r)|^{1+\varepsilon}dr-\int_{\overline s}^s|y(r;x)|^{1+\varepsilon}dr\right )>0,
\end{eqnarray*}
which, in its turn, imply that $\zeta<y(\cdot;x)$ in $(\overline s,s_0)$: a contradiction.
\end{proof}

\section{Gradient estimates}
\label{sect_graest}
This section is devoted to establish some uniform gradient estimates for the function $G_{\mathcal B}(t,s)f$.
More precisely, our aim consists in proving that, for any $T>s\in I$, there exists a positive constant $C_{s,T}$ such that
\begin{equation}\label{est_neu}
\|\nabla_x G_{\mathcal B}(t,s)f\|_{\infty} \le \frac{C_{s,T}}{\sqrt{t-s}}\|f\|_{\infty},\qquad\;\, t\in (s,T),
\end{equation}
for any $f\in C_b(\Omega)$. In the particular case when $C_{s,T}\le C(s)$ for some function $C$ bounded from above in any right-halfline
$J\subset I$, estimate \eqref{est_neu} allows us to conclude that, for any $\varepsilon>0$, there exists $C_{s,\varepsilon}'>0$ such that
\begin{equation}\label{est_neu-neu}
\|\nabla_x G_{\mathcal B}(t,s)f\|_{\infty} \le \frac{C'_{s,\varepsilon}}{\sqrt{t-s}}e^{-(c_0-\varepsilon)(t-s)}\|f\|_{\infty},
\qquad\;\, t\in (s,+\infty),
\end{equation}
for the same $f$'s as above. Indeed, in this case,
\begin{equation}
\|\nabla_xG_{\mathcal B}(t,r)f\|_{\infty}\le \frac{C(r)}{\sqrt{t-r}}\|f\|_{\infty}\le
\frac{C_s}{\sqrt{t-r}}\|f\|_{\infty},\qquad\;\,s\le r<t\le r+1,
\label{star-star}
\end{equation}
for any $f\in C_b(\Omega)$, where $C_s=\sup_{r>s}C(r)$.
Now, if $t>s+1$, we split $G_{\mathcal B}(t,s)f=G_{\mathcal B}(t,t-1)G_{\mathcal B}(t-1,s)f$ and use \eqref{est_neu} to estimate
\begin{align*}
\|\nabla_xG_{\mathcal B}(t,s)f\|_{\infty}\le &C_s\|G_{\mathcal B}(t-1,s)f\|_{\infty}
\le\tilde C_se^{-c_0(t-s)}\|f\|_{\infty},
\end{align*}
which, combined with \eqref{star-star}, yields to \eqref{est_neu-neu}.
\medskip

Throughout this section, besides Hypotheses \ref{hyp1}, \ref{hyp1-bc} and \ref{hyp-3} we will consider the following conditions on the domain $\Omega$ and the coefficients of the operators
$\A$ and ${\mathcal B}$. In particular, we assume that the boundary operator is independent of $t$.
\begin{hyp}
\label{hyp-reg-coeff}
\begin{enumerate}[\rm (i)]
\item
$\partial\Omega$ is uniformly of class $C^{3+\alpha}$;
\item
$q_{ij}, b_j, c\in C^{\alpha/2,1+\alpha}_{\rm loc}(\Omega_I)$ for some $\alpha\in (0,1)$
and any $i,j=1,\ldots,d$;
\item
there exist locally bounded from above functions $L_j, M_1:I\to\R$ $(j=1,\ldots,4)$, with $L_1$, $L_2$ nonnegative and $L_4<1/2$ in $I$, such that
\begin{equation}\label{der_q}
\hskip .5truecm
(a)~|\nabla_x q_{ij}(t,x)|\le M_1(t)\eta(t,x),\qquad\;\, (b)~|\nabla_x c(t,x)| \le L_1(t)+L_2(t)c(t,x),
\end{equation}
and
\begin{equation}\label{dissip_c}
\langle J_x b(t,x) \xi, \xi\rangle \le (L_3(t)+L_4(t)c(t,x))|\xi|^2,\qquad \;\, \xi\in\Rd,
\end{equation}
for any $(t,x)\in\Omega_I$;
\item
for any bounded interval $J\subset I$, there exists $\delta_0=\delta_0(J)$ such that $q_{ij}$, $b_j$ and $c$
belong to $C^{0,\alpha}_b(J\times\Omega_{\delta_0})$;
\item
either $(\beta, \gamma)\equiv (0,1)$ or $\beta\in C^{2+\alpha}_{\rm loc}(\partial \Om;\Rd)$ is bounded together with its derivatives, satisfies $\inf_{x \in \partial\Omega}\langle\beta(x),\nu(x)\rangle>0$ and $\gamma\in C_b(\partial\Omega)\cap C^{1+\alpha}_{\rm loc}(\partial\Omega)$.
\end{enumerate}
\end{hyp}

\begin{rmk}
\label{reg}
{\rm
Note that it is enough to prove estimate \eqref{est_neu} for functions $f\in C^3_c(\Om)$. Indeed, if $f\in C_b(\Om)$
we can find a sequence $(f_n)\in C^3_c(\Om)$ converging to $f$ locally uniformly in $\Om$ and such that $\|f_n\|_{\infty}\le\|f\|_{\infty}$ for any
$n\in\N$. By Proposition \ref{continuity}(i), $\nabla_xG_{\mathcal B}(\cdot,s)f_n$ converges pointwise to $\nabla_xG_{\mathcal B}(\cdot,s)f$
in $\Om_{(s,T)}$. Hence, from \eqref{est_neu}, with $f$ being replaced by $f_n$, we get
\begin{eqnarray*}
|(\nabla_xG_{\mathcal B}(t,s)f_n)(x)|\le \frac{C_{s,T}}{\sqrt{t-s}}\|f_n\|_{\infty}\le\frac{C_{s,T}}{\sqrt{t-s}}\|f\|_{\infty}.
\end{eqnarray*}
Letting $n\to +\infty$, we obtain \eqref{est_neu} for $f\in C_b(\Om)$.}
\end{rmk}

In view of this remark, we will prove \eqref{est_neu} for functions $f\in C^{3+\alpha}_c(\Omega)$.

\begin{thm}
\label{thm-grad-gen}
Under Hypotheses $\ref{hyp-reg-coeff}$, estimate \eqref{est_neu} holds true, with the constant $C_{s,T}$ depending on $d$, $\eta_0$,
$\|q_{ij}\|_{C^{0,\alpha}_b((s,T)\times\Omega_{\delta_0})}$, $\|b_j\|_{C^{0,\alpha}_b((s,T)\times\Omega_{\delta_0})}$ $(i,j=1,\ldots,d)$,
$\|c\|_{C^{0,\alpha}_b((s,T)\times\Omega_{\delta_0})}$, $\sup_{(s,T)}L_j$ $(j=1,2,3,4)$ and $\sup_{(s,T)}M_1$.
If all the functions $L_j$ $(j=1,2,3)$ and $M_1$ are bounded from above in $(s,+\infty)$ and $\sup_{(s,+\infty)}L_4<\frac{1}{2}$,
then estimate \eqref{est_neu-neu} holds true for any $\varepsilon>0$, and the constant therein appearing is independent of $s$
if in addition the functions $M_1$, $L_j$ $(j=1,2,3)$ are bounded from above in $I$, $\sup_IL_4<\frac{1}{2}$
and $q_{ij}$, $b_j$ $(i,j=1,\ldots,d)$ belong to $C^{0,\alpha}_b(I\times\Omega_{\delta_0})$.
\end{thm}
\begin{proof}
Fix $T>s\in I$ and $f \in C^3_c(\Rd)$. We split the proof into two steps.
In the first one, we prove a uniform gradient estimate for $G_{\mathcal B}(t,s)f$ near the boundary of $\Omega$. More precisely, we prove estimate \eqref{est_neu} with $\Omega$ being replaced by $\Omega_{\delta_1}$ for a suitable $\delta_1>0$.
Here, the smoothness of the domain suggests to go back, by means of local charts (and Lemma \ref{geometric}), to smooth bounded domains of $\Rd_+$
and to consider problems therein defined. In the second step, we prove an interior gradient estimate, i.e., we show that estimate
\eqref{est_neu} is satisfied with $\Omega$ being replaced by $\Omega\setminus\Omega_{\delta_1}$. Clearly, combining the results in Steps 1 and 2,
\eqref{est_neu} follows at once.

Throughout the proof, we denote by $C$ positive constants, which are independent of $n$ and $x_0\in\partial\Omega$, which may vary from line to line.

{\em Step 1.} We first consider the case when $\B$ is a first-order boundary operator.
We fix $0<\delta_1<\min\{\delta_0,r_0\}$, where $r_0$ is as in Lemma \ref{geometric} and $\delta_0=\delta_0((s,T))$ is given by
Hypothesis  \ref{hyp-reg-coeff}(iv), and prove estimate \eqref{est_neu}, with
$\Omega$ being replaced by $\Omega_{\delta_1}$.
Clearly, since $\bigcup_{x_0\in\partial\Omega}B_{\delta_1}(x_0)=\Omega_{\delta_1}$, it suffices to prove that there exists a positive constant
$K_{s,T}$, independent of $x_0$, such that
\begin{equation}\label{moscone}
\|\nabla_x G_{\mathcal B}(t,s)f\|_{L^\infty( \Om\cap B_{\delta_1}(x_0))}\le \frac{K_{s,T}}{\sqrt{t-s}}\|f\|_{\infty},
\qquad\;\, t\in (s,T),\;\, x_0\in\partial\Omega.
\end{equation}

Fix $x_0\in\partial\Omega$, $r_1\in (\delta_1,r_0)$ and  define
$R_n= 2\delta_1-r_1+(r_1-\delta_1)\sum_{k=0}^n 2^{-k}$ for any $n\in\N\cup\{0\}$.
Using Lemma \ref{lemma-approx}, we determine a sequence $(\vartheta_n)\subset C^{\infty}_c(\Rd)$ such that
$\chi_{\phi_{x_0}(B_{R_n}(x_0)\cap\Omega)}\le\vartheta_n\le \chi_{\phi_{x_0}(B_{R_{n+1}}(x_0)\cap\overline\Omega)}$ ($\phi_{x_0}$ is as in Lemma
\ref{geometric}), $D_d\vartheta_n\equiv 0$ on $\partial\Rd_+$ and
\begin{equation}
\|\vartheta_n\|_{C^k_b(\Rd)}\le \frac{2^{kn}C}{(r_1-\delta_1)^k},\qquad\;\,k=1,2,3.
\label{Dk-vartheta-1}
\end{equation}
Again by Lemma \ref{lemma-approx}, we fix a smooth function $\zeta$ such that $\chi_{\phi_{x_0}(B_{R_{n+1}}(x_0)\cap \Om)}\le\zeta\le
\chi_{\phi_{x_0}(B_{r_0}(x_0)\cap\Omega)}$. Since the support of the function
$w_n=\vartheta_nv_{x_0}:=\vartheta_n(G_{\mathcal B}(\cdot,s)f)(\phi_{x_0}^{-1})$ is contained in $\phi_{x_0}(B_{R_{n+1}}(x_0)\cap\Omega)$, a long but
straightforward computation reveals that $w_n$ solves the Cauchy problem
\begin{equation*}
\left\{
\begin{array}{ll}
D_t w_n(t,x)= (\hat{\A}w_n)(t,x)+ \hat{g}_n(t,x), \quad & (t,x) \in (s,T)\times \Rd_+,\\[1.5mm]
D_dw_n(t,x)+\omega(x)w_n(t,x)=0,\quad & (t,x)\in (s,T)\times\partial\Rd_+\\[1.5mm]
w_n(s,x)= \hat{f_n}(x), \quad & x \in \Rd_+,
\end{array}
\right.
\end{equation*}
where $\hat{\A}={\rm Tr}(\hat{Q}D^2)+\langle \hat{b},\nabla_x\rangle-\hat{c}$,
with $\hat{Q}= \zeta J\phi_{x_0}(\phi_{x_0}^{-1})Q(\cdot,\phi_{x_0}^{-1})(J\phi_{x_0}(\phi_{x_0}^{-1}))^T+(1-\zeta)I$,
$\hat{b}=\zeta[ (J\phi_{x_0}(\phi_{x_0}^{-1})b(\cdot,\phi_{x_0}^{-1}))_h+{\rm Tr}(Q(\cdot,\phi_{x_0}^{-1})D^2\phi^h_{x_0}(\phi_{x_0}^{-1}))]$
($h=1,\ldots,d$), $\hat{c}= \zeta c(\cdot,\phi_{x_0}^{-1})$,
$\omega=\zeta\gamma(\phi_{x_0}^{-1})/\rho_{x_0}(\phi_{x_0}^{-1})$ ($\rho_{x_0}$ is defined in \eqref{costarica}).
Finally, $\hat{g}_n=-2\langle \hat{Q}\nabla\vartheta_n, \nabla_xv_{x_0}\rangle - v_{x_0}(\hat{\A}+\hat{c})\vartheta_n$ and
$\hat{f}_n=\vartheta_n f(\phi_{x_0}^{-1})$, defined in the whole of $\Rd_+$.
Note that the coefficients of the operator $\hat\A$ and the function $\omega$ are smooth and bounded.

Denote by $G_{\mathcal R}(t,s)$ the evolution operator associated to $\hat{\A}$ in $C_b(\Rd_+)$ with homogeneous Robin boundary conditions.
Using the optimal Schauder estimates
$\|G_{\mathcal R}(t,s)\psi\|_2\le C(t-s)^{-\frac{3}{4}}\|\psi\|_{1/2}$, which holds for any $t \in (s, T]$
and $\psi \in C^{1/2}_b(\Rd_+)$
(where, from now on, we simply write $\|\cdot\|_{\beta}$ to denote the norm in $C^{\beta}_b(\Rd_+)$)
and the variation-of-constants formula, we can estimate
\begin{align}
&(t-s)\|D^2_xw_n(t,\cdot)\|_{\infty}\notag\\
\le & (t-s)\left [\|D^2_xG_{\mathcal R}(t,s)\hat{f_n}\|_{\infty}
+\left\|\int_s^t (D^2_xG_{\mathcal R}(t,r)\hat g_n(r,\cdot))(\cdot)\,dr\right\|_{\infty}\right ]\notag\\
\le & C\bigg\{\|f\|_{\infty}+ (t-s)\int_s^t (t-r)^{-\frac{3}{4}}\|\hat g_n(r,\cdot)\|_{\frac{1}{2}}dr\bigg\},
\label{gabriella}
\end{align}
for any $t \in (s, T]$. Since $w_{n+1}\equiv v_{x_0}$ in
$\phi_{x_0}(B_{R_{n+1}}(x_0)\cap \Omega)$ and $\hat g_n(r,\cdot)$ is
supported in $\phi(B_{R_{n+1}}(x_0)\cap \Omega)$, for any $r\in (s,T)$ we have
\begin{align*}
\|\hat g_n(r,\cdot)\|_{\frac{1}{2}}\le &C\|\vartheta_n\|_{C^3_b(\Rd)}\Big(\|\nabla_x w_{n+1}(r,\cdot)\|_{\frac{1}{2}}
+\|w_{n+1}(r,\cdot)\|_{\frac{1}{2}}\Big)\notag\\
\le&  8^{n}C\Big ((r-s)^{-\frac{3}{4}}\sup_{\sigma \in (s,T)}(\sigma-s)^{\frac{3}{4}}\|\nabla_x w_{n+1}(\sigma, \cdot)\|_{\frac{1}{2}}
+\|f\|_{\infty}\Big ),
\end{align*}
where the constant $C$ depends on $\|q_{ij}\|_{C^{0,\alpha}_b((s,T)\times\Omega_{\delta_0})}$ and $\|b_j\|_{C^{0,\alpha}_b((s,T)\times\Omega_{\delta_0})}$
($i,j=1,\ldots,d$).
Here, we have used the estimate $\|w_{n+1}(r,\cdot)\|_{\frac{1}{2}}\le 3\|w_{n+1}(r,\cdot)\|_{\infty}+\|\nabla_x w_{n+1}(r,\cdot)\|_{\frac{1}{2}}$
and \eqref{Dk-vartheta-1}. Thus, it follows that
\begin{align}\label{conto3}
(t-s)\int_s^t(t-r)^{-\frac{3}{4}}\|\hat g_n(r,\cdot)\|_{\frac{1}{2}}dr
\le 8^{n}C\Big (\sup_{\sigma \in (s,T)}(\sigma-s)^{\frac{3}{4}}\|\nabla_x w_{n+1}(\sigma, \cdot)\|_{\frac{1}{2}}\!+\!\|f\|_{\infty}\Big ),
\end{align}
for any $t\in (s,T)$.
Since $\|\nabla_xw_{n+1}(r,\cdot)\|_{\frac{1}{2}}\le C\|w_{n+1}(r,\cdot)\|_{\infty}^{\frac{1}{4}}\|D_x^2w_{n+1}(r,\cdot)\|_{\infty}^{\frac{3}{4}}$,
using estimate \eqref{stima-oper-B} and Young inequality, we deduce that
\begin{align}
\sup_{r \in (s,T)}(r-s)^{\frac{3}{4}}\|\nabla_x w_{n+1}(r,\cdot)\|_{\frac{1}{2}}
\le Ca_{n+1}^{\frac{3}{4}}\|w_{n+1}\|_{\infty}^{\frac{1}{4}}
\le \varepsilon a_{n+1}+C\varepsilon^{-3}\|f\|_{\infty},
\label{conto-3bis}
\end{align}
for any $\varepsilon >0$ and $n\in \N\cup\{0\}$, where $a_k:=\sup_{t \in (s, T)}(t-s)\|D^2 w_k(t,\cdot)\|_{\infty}$ for any $k\in\N\cup\{0\}$.
Now, replacing \eqref{conto3} and \eqref{conto-3bis} in \eqref{gabriella} we obtain
\begin{align}\label{primavera11}
a_n\le &8^nC\big (\varepsilon a_{n+1}+\varepsilon^{-3}\|f\|_{\infty}\big ),\qquad\;\,n\in\N\cup\{0\},\;\,\varepsilon>0.
\end{align}

The classical Schauder estimates in \cite[Thm. IV.10.1]{LadSolUra68Lin} and \eqref{stima-oper-B} show that
$\|v_{x_0}(r,\cdot)\|_{C^2(\phi(B_{r_1}(x_0)\cap \Om))}\le C\|f\|_{\infty}$, for any $r\in (s,T)$ where $C$ depends also on $\|c\|_{C^{0,\alpha}_b((s,T)\times\Omega_{\delta_0})}$. It thus follows that
$a_n \le 4^nC\|f\|_{\infty}$ for any $n \in \N\cup\{0\}$. We can
now choose $\varepsilon>0$ in \eqref{primavera11} such that
$\tau:=\varepsilon 8^nC< 2^{-9}$. Multiplying both the sides of \eqref{primavera11} by $\tau^n$ and summing over $n\in\N$, we
realize that the two series converge (in view of the above estimate on $a_n$) and, as a by product, we deduce that
\begin{equation}
(t-s)\|D^2_xw_0(t,\cdot)\|_{\infty}\le C\|f\|_{\infty},\qquad\;\,t\in (s,T).
\label{baglioni}
\end{equation}
Since $\|\nabla_xw_0(t,\cdot)\|_{\infty}\le C\|w_0(t,\cdot)\|_{\infty}^{1/2}\|D^2_xw_0(t,\cdot)\|_{\infty}^{1/2}$,
from \eqref{baglioni} it follows immediately that $\sqrt{t-s}\,\|\nabla_xw_0(r,\cdot)\|_{\infty}\le C\|f\|_{\infty}$ for any $t\in (s,T)$.
Recalling that $\vartheta_0\equiv 1$ in $\phi(B_{\delta_1}(x_0)\cap\Omega)$ we conclude that
$\sup_{t \in (s, T)}\sqrt{t-s}\,\|\nabla_x v_{x_0}(t,\cdot)\|_{\phi(B_{\delta_1}(x_0)\cap\Omega)}\le C\|f\|_{\infty}$.
Now, taking into account that $\nabla_x v_{x_0}= (J \phi_{x_0}^{-1})^T(\nabla_x G_{\mathcal B}(\cdot,s)f)(\phi_{x_0}^{-1})$, estimate \eqref{moscone}
follows at once.

In the case of homogeneous Dirichlet boundary conditions, the proof is completely similar. Actually, Lemma \ref{geometric} is not
needed here, since one can use the covering $\{\psi_h: h\in\N\}$ of $\partial\Omega$.

{\em Step 2.} We fix two functions $\vartheta_1, \vartheta_2 \in C^{\infty}(\Rd)$ such that
$\chi_{\Omega\setminus\Omega_{\delta_1}}\le\vartheta_1\le\chi_{\Om\setminus\Omega_{\delta_1/2}}$
and $\chi_{\Omega\setminus\Omega_{\delta_1/2}}\le\vartheta_2\le\chi_{\Om\setminus\Omega_{\delta_1/4}}$.
We denote by $v$ the trivial extension to the whole of $\Rd$ of the function
$\vartheta_1G_{\mathcal B}(\cdot,s)f$. As it is easily seen, the function $v$ solves the Cauchy problem
\begin{equation*}
\left\{
\begin{array}{ll}
D_tv(t,x)=(\tilde\A v)(t,x)+\psi(t,x), &t\in(s,+\infty),\,x\in\Rd,\\[1mm]
v(s,x)= \tilde f(x),& x\in \Rd,
\end{array}
\right.
\end{equation*}
where $\psi$ (resp. $\tilde f$) is the trivial extension to the whole of $(s,+\infty)\times\Rd$ (resp. $\Rd$) of the function
$\psi=-(G_{\mathcal B}(\cdot,s)f)(\A+c)\vartheta_1 -2 \langle Q\nabla_xG_{\mathcal B}(\cdot,s)f,\nabla \vartheta_1\rangle$
(resp. $\vartheta_1 f$) and $\tilde\A(t)={\rm Tr}(\tilde Q(t,\cdot)D^2)+\langle \tilde b(t,\cdot),\nabla_x\rangle+\tilde c(t,\cdot)$, where
$\tilde Q=\vartheta_2Q+(1-\vartheta_2)I$, $\tilde b=\vartheta_2 b$ and $\tilde c=\vartheta_2 c$.
Since the continuous function $\psi$ is supported in $\Omega_{\delta_1}$, in view of the boundedness assumptions on the diffusion
and drift coefficients, the definition of the function $\vartheta_1$ and Step 1,
\begin{equation}
\|\psi(t,\cdot)\|_{\infty}\le \frac{C}{\sqrt{t-s}}\|f\|_{\infty},\qquad\;\,t\in (s,T).
\label{luca-lunedi-0}
\end{equation}
Therefore, arguing as in Step 1, we can easily show that
\begin{equation}
\nabla_xv(t,x)=(\nabla_xG(t,s)\tilde f)(x)+\int_s^t(\nabla_xG(t,r)\psi(r,\cdot))(x)dr,
\label{luca-lunedi}
\end{equation}
for any $(t,x)\in (s,T)\times\Rd$, where $G(t,s)$ denotes the evolution operator
associated to the operator $\tilde{\mathcal A}$ in $C_b(\Rd)$ (see \cite{AngLor10Com}).

We claim that there exists a positive constant $C$, independent of $f$, such that
\begin{equation}
|(\nabla_xG(t,s)g)(x)|\le \frac{C}{\sqrt{t-s}}\|g\|_\infty,\qquad\;\,t\in (s,T),\;\,x\in\Rd,
\label{stima-grad-spazio}
\end{equation}
for any $g\in C_c(\Rd)$. Once this estimate is proved, from \eqref{luca-lunedi-0} and \eqref{luca-lunedi} it follows that
$\sqrt{t-s}\|\nabla_xv(t,\cdot)\|_{\infty}\le C\|f\|_{\infty}$
for any $t>s$, from which the gradient estimate for $G_{\mathcal B}(t,s)f$ in $\Omega\setminus\Omega_{\delta_1}$ follows immediately, recalling that
$v\equiv G_{\mathcal B}(\cdot,s)f$ in $(s,+\infty)\times (\Omega\setminus\Omega_{\delta_1})$.

To prove \eqref{stima-grad-spazio}, we fix $g\in C_c(\Rd)$ and, for any $n\in\N$ such that ${\rm supp}(g)\subset B_n$, we introduce
the evolution operator $G_n^N(t,s)$ associated to $\tilde\A$, with homogeneous Neumann boundary conditions, in $C_b(B_n)$.
By \cite[Thm. 2.3]{AngLor10Com}
$\nabla_x G_n^N(\cdot,s)g$ converges to $\nabla_x G(\cdot,s)g$ pointwise in $(s,T]\times \Rd$. Let $z_n\in C_b([s,T]\times\overline B_n)
\cap C^{1,2}((s,T)\times B_n)$ be the function defined by
$z_n(t,x):= (u(t,x))^2+a(t-s)|(\nabla_xu(t,x)|^2$ for any $(t,x)\in [s,T]\times B_n$,
where $u:=G_n^N(\cdot,s)g$ and the constant $a$ will be chosen later on. Since the matrix $J\nu$ is positive definite, the normal derivative
of $z_n$ is nonpositive on $\partial B_n$ (see the proof of Theorem \ref{caso_convesso} for further details). A simple computation shows that $z_n$ satisfies problem
\begin{align*}
\left\{
\begin{array}{ll}
D_tz_n(t,x)=(\tilde{\mathcal A}z_n)(t,x)+\psi_n(t,x),\qquad & t\in(s,T], x \in B_n\\[1mm]
\displaystyle\frac{\partial z_n}{\partial \nu}(t,x)\le 0,\qquad & t\in (s,T], x \in \partial B_n,\\[2mm]
z_n(s,x)=(g(x))^2,\qquad& x \in B_n,
\end{array}
\right.
\end{align*}
where
\begin{align}
\psi_n=& a|\nabla_x u|^2-\tilde c u^2
-2\langle \tilde Q\nabla_x u, \nabla_x u\rangle- a(\cdot-s)\tilde c|\nabla_x u|^2
-2a(\cdot-s){\rm Tr}(D^2_xu\tilde QD^2_xu)\notag\\
&+2a(\cdot-s)\sum_{i,j,k=1}^dD_k\tilde q_{ij}D_k uD_{ij}u
+2a(\cdot-s)[\langle J_x\tilde b\,\nabla_xu,\nabla_x u\rangle-u\langle \nabla_x \tilde c,\nabla_x u\rangle].
\label{ternoto}
\end{align}
Notice that the coefficients of the operator $\tilde{\A}(t)$ satisfy Hypothesis \ref{hyp-reg-coeff}(iii)
with the same values of $L_2$ and $L_4$ and with $M_1$, $\eta$, $L_1$ and $L_3$ being replaced, respectively, by
$\tilde M_1=M_1\chi_{\Om}+\max\{1,\eta_0^{-1}\}\|\nabla\vartheta_2\|_{\infty}(\|q_{ij}\|_{C_b((s,T)\times\Omega_{\delta_0})}+1)\chi_{\Om_{\delta_1}}$,
$\tilde\eta=\vartheta_2 \eta+1-\vartheta_2$, $\tilde L_1=\vartheta_2L_1+\|\nabla\vartheta_2\|_{\infty}\|c\|_{C_b((s,T)\times\Omega_{\delta_0})}$ and
$\tilde L_3=\vartheta_2 L_3+\|\nabla\vartheta_2\|_{\infty}\|b\|_{C_b((s,T)\times\Omega_{\delta_0})}$.
Therefore, we can estimate
\begin{equation}
\langle \tilde Q\nabla_x u,\nabla_x u\rangle\ge\tilde\eta|\nabla_xu|^2,\qquad\;\,
{\rm Tr}(D^2_xu(t,\cdot)\tilde Q(t,\cdot)D^2_xu(t,\cdot))\ge\tilde\eta |D^2_xu|^2,
\label{luca-1}
\end{equation}
\begin{equation}
\bigg |\sum_{i,j,k=1}^dD_k\tilde q_{ij}D_k u D_{ij}u\bigg |\le \tilde{M}_1d\tilde\eta|\nabla u||D^2_xu|,
\qquad |\langle\nabla_x \tilde c,\nabla_x u\rangle |\le (\tilde L_1+L_2\tilde c)|\nabla_xu|,
\label{luca-2}
\end{equation}
\begin{equation}
\langle J_x\tilde b\nabla_xu,\nabla_xu\rangle\le (\tilde L_3+L_4\tilde c)|\nabla_xu|^2,
\label{luca-4}
\end{equation}
in $\Omega_{(s,T)}$. Now, estimating
\begin{eqnarray*}
|\nabla_xu||D^2_xu|\le\varepsilon |D^2_xu|^2+\frac{1}{4\varepsilon}|\nabla_xu|^2,\qquad
|u||\nabla_xu|\le \varepsilon|\nabla_xu|^2+\frac{1}{4\varepsilon}u^2,
\end{eqnarray*}
for any $\varepsilon>0$, from \eqref{ternoto}-\eqref{luca-4} we deduce that
\begin{align*}
\psi_n\le & \frac{a}{2\varepsilon}(T-s)\overline{L}_1u^2
+\left (\frac{a}{2\varepsilon}(T-s)\overline{L}_2-1\right )\tilde cu^2\notag\\
&+\bigg [a+\left (a(T-s)\frac{\overline{M}_1d}{2\varepsilon}-2\right )\tilde\eta
+2a(T-s)\left (\overline{L}_3+\overline{L}_1\varepsilon\right )\bigg ]|\nabla_x u|^2\notag\\
&+a(T-s)\left (2\varepsilon \overline{L}_2+2\overline{L}_4-1\right )^+\tilde c|\nabla_xu|^2
+2a(T-s)\left (\overline{M}_1d\varepsilon-1\right )^+\tilde \eta|D^2_x u|^2,
\label{acqua_fonte-10}
\end{align*}
for any $\varepsilon>0$, where $\overline{L_{2j-1}}=\sup_{(s,T)}\tilde L_{2j-1}$,
$\overline{L_{2j}}=\sup_{(s,T)}L_{2j}$ ($j=1,2$) , $\overline{M_1}=\sup_{(s,T)}\tilde M_1$.
Thus, choosing $\varepsilon=\frac{1}{2}\min\left\{\frac{1-2\overline L_4}{\overline{L}_2},\frac{1}{\overline{M}_1d}\right\}$,
we can make nonpositive the coefficients in front of both $\tilde c|\nabla_xu|^2$ and $|D^2_xu_x|^2$.
Observing that the coefficients in front of $\tilde cu^2$ and $|\nabla_xu|^2$ tend, respectively, to
$-1$ and $-2\tilde\eta<-2\tilde\eta_0$, as $a\to 0^+$, we can then choose $a$ small enough such that
these coefficients are negative. With these choices of $\varepsilon$ and $a$, we deduce that
$\psi_n\le H_{s,T}u^2\le H_{s,T}z_n$ for any $n\in\N$ and some positive constant $H_{s,T}$,
depending on $\overline{L}_j$ $(j=1,2,3,4)$, $\overline{M}_j$ ($j=1,2$) $\eta_0$, $d$, $s$ and $T$.
By applying the classical maximum principle to the function $(t,x)\mapsto e^{-H_{s,T}(t-s)}z_n(t,x)$ we conclude that
 $e^{-H_{s,T}(t-s)}z_n \le \|g\|_{\infty}^2$, i.e.,
\begin{eqnarray*}
((G_n^N(t,s)g)(x))^2+a(t-s)|(\nabla_xG_n^N(t,s)g)(x)|^2\le C\|g\|_{\infty}^2,\qquad (t,x)\in [s,T]\times B_n.
\end{eqnarray*}
Letting $n \to +\infty$ we get \eqref{stima-grad-spazio}.
\end{proof}

In the following subsection, we consider the particular cases when the operator ${\mathcal A}$ is endowed with Neumann and Robin boundary conditions. In the first case we show that the boundedness assumptions on its coefficient in a neighborhood of $\partial\Omega$ and the additional smoothness condition on
$\Omega$ can be removed provided that $\Omega$ is convex.

\subsection{Neumann boundary conditions}

\begin{thm}\label{caso_convesso}
Let $\Om$ be a convex open set. Then, under Hypotheses $\ref{hyp-reg-coeff}(ii), (iii)$,
estimate \eqref{est_neu} holds true with the constant $C_{s,T}$ depending also on
$\sup_{(s,T)}M_1$, $\sup_{(s,T)}L_j$ $(j=1,2,3,4)$.
If the functions $L_j$ $(j=1,2,3)$ and $M_1$ are bounded from above in $(s,+\infty)$
and $\sup_{(s,+\infty)}L_4<\frac{1}{2}$, then estimate \eqref{est_neu-neu} holds true
for any $\varepsilon>0$, and the constant therein appearing is independent of $s$ if further
$L_j$ $(j=1,2,3)$, $M_1$ are bounded from above in $I$ and $\sup_{I}L_4<\frac{1}{2}$.
\end{thm}
\begin{proof}
The proof is an adaption to the nonautonomous case of the gradient estimates in \cite{BerFor04Gra}.

Fix $T>s\in I$, $f\in C^3_c(\Omega)$ and an increasing sequence $(\Omega_n)$ of bounded, smooth convex sets such that
$\lim_{n\to +\infty}\Omega_n=\Omega$ and $\partial\Omega\cap\partial\Omega_n\neq\varnothing$ for any $n\in\N$.
Denote by $G_n^N(t,s)$ the evolution operator in $C_b(\Omega_n)$ associated with $\A$ with homogeneous Neumann
boundary conditions on $\partial\Omega_n$. Adapting the arguments in the proof of
Theorem \ref{existence}, we can easily prove that $G_n^N(\cdot,s)f$ converges to $G_{\mathcal B}(\cdot,s)f$
in $C^{1,2}(K)$ for any compact set $K\subset \overline{\Omega_{(s,+\infty)}}$. Since the normal derivative
of $G_n^N(t,s)f$ identically vanishes on $\partial\Omega_n$, each tangential derivative on $\partial\Omega$ of
$\frac{\partial}{\partial\nu}G_n^N(t,s)f$ vanishes. Therefore,
$\langle (D^2G_n^N(t,s)f)(x)\nu(x),\tau\rangle+\langle J\nu(x)\tau,(\nabla_xG_n^N(t,s)f)(x)\rangle=0$
for any vector $\tau$ tangent to $\partial\Omega$ at $x$, and any $x\in\partial\Omega$.
In particular, taking $\tau=(\nabla_x G_n^N(t,s)f)(x)$ and recalling that, since $\Omega_n$ is convex,
the quadratic form associated with the matrix $J\nu$ is everywhere nonnegative on $\partial\Omega_n$,
we conclude that $\langle (D^2G_n^N(t,s)f)(x)(\nabla_xG(t,s)f)(x), \nu(x)\rangle \le 0$ for any $x\in\partial\Omega_n$.
Therefore, the function $|\nabla_xG_n^N(t,s)f|^2$ has nonpositive normal derivative
on $\partial\Omega$.
As a byproduct, for any $n\in\N$ the function $z_n=|G_n^N(\cdot,s)f|^2+a(\cdot-s)|\nabla_xG_n^N(\cdot,s)f|^2$ has a nonpositive normal derivative
on $\partial\Omega_n$. We can now argue as in Step 2 of the proof of Theorem \ref{thm-grad-gen} and show that, for a suitable choice of the parameter
$a$, the function $D_tz_n-\A z_n$ is nonpositive in $(s,T)\times\Omega_n$. Hence, using the classical maximum principle and letting $n\to +\infty$,
we obtain estimate \eqref{est_neu}.
\end{proof}

As a consequence of Theorem \ref{caso_convesso} we can prove gradient estimates for solutions to problem \eqref{NAPC} in $\Rd_+$ when
${\mathcal A}$ is endowed with Robin boundary conditions, i.e., when ${\mathcal B}={\mathcal R}=\frac{\partial }{\partial\nu}+\gamma I$.
Besides Hypotheses \eqref{hyp-reg-coeff}(ii), (iii) we assume the following conditions:

\begin{hyp}
\label{hyp-10}
\begin{enumerate}[\rm (i)]
\item
the diffusion coefficients $q_{ij}$ belongs to $C^{0,1}_b(J\times\Rd_{+,\delta})$ for some $\delta>0$ and any bounded interval $J\subset I$;
\item
there exists a locally bounded function $L_5:I \to (0,+\infty)$ such that $|b|\le L_5(1+c)$ in $I\times\Rd_{+,\delta}$;
\item
$\gamma \in C^{2+\alpha}_{\rm loc}({\overline{\Rd_+}})$ and
there exist a constant $L_6$, a locally bounded from above function $L_7:I\to (0,+\infty)$ and a function $\Gamma\in C^{3+\alpha}_{\rm loc}(\overline{\Rd_+})$
such that $\Gamma$ is supported in $\overline{\Rd_{+,\delta}}$,
$D_d\Gamma\equiv\gamma$ on $\partial\Rd_+$,
\begin{equation}
\|\nabla\Gamma\|_{\infty}+\|D^2\Gamma\|_{\infty}+\|D^3\Gamma\|_{\infty}\le L_6,
\label{baglioni-1}
\end{equation}
\begin{equation}
\inf_{\Rd_{+,\delta}}\left [(\A+c)\Gamma-\langle Q\nabla\Gamma,\nabla\Gamma\rangle\right ]\ge -L_7,
\label{limit-alto}
\end{equation}
in $I$, where $\delta$ is as in $(i)$.
\end{enumerate}
\end{hyp}

\begin{rmk}
{\rm Sufficient conditions for Hypothesis \ref{hyp-10}(iii) hold are the following:
\begin{enumerate}[\rm (i)]
\item
the support of $\gamma$ is contained in a compact set $K\subset\R^{d-1}$. In this case we can take
$\Gamma(x)=\gamma(x')\vartheta(x_d)$ for any $x\in\Rd_+$,
where $\vartheta$ is a smooth nonnegative function supported in $[0,\delta]$ such that $\vartheta\le 1$ in $[0,\delta]$ and $\vartheta'(0)=1$;
\item
$\gamma(x)=\gamma_1(|x'|^2)$ for any $x\in\R^d_+$, where  $\gamma_1$ is a bounded and not increasing smooth
function such that $\gamma_0:=\sup_{t\ge 0}(1+t)^k(|\gamma_1'(t)|+|\gamma''_1(t)|)<+\infty$ for some  $k\in\N$. Further, there exists
a positive locally bounded function $L:I\to\R$ such that $|q_{ij}(t,x)|\le L(t)(1+|x'|^2)^{k-1}$,
 $|q_{id}(t,x)|\le L(t)(1+|x'|^2)^{k-1/2}$ ($i,j<d$) and $\langle b'(t,x),x'\rangle\le L(t)(1+|x'|^2)^k$ for any $(t,x)\in I\times\Rd_{+,\delta}$
and some $\delta>0$, where $b=(b',b_d)$.
 Finally, $q_{dd}$ and $b_d$ are bounded in $\Rd_{+,\delta}$. Indeed, in this case, with the choice of $\Gamma$ as in (i), we get
\begin{align*}
&((\A+c)\Gamma)-\langle Q\nabla\Gamma,\nabla\Gamma\rangle\\
\ge & -4d\gamma_0L-4(d-1)\gamma_0^2L-4\sqrt{d-1}(1+\|\gamma_1\|_{\infty})\gamma_0\|\vartheta'\|_{\infty}L\\
&-\|b_d\|_{\infty}\|\gamma_1\|_{\infty}\|\vartheta'\|_{\infty}-\|q_{dd}\|_{\infty}\|\gamma_1\|_{\infty}(\|\gamma_1\|_{\infty}\|\vartheta'\|_\infty^2+\|\vartheta''\|_{\infty}),
\end{align*}
in $I\times\Rd_{+,\delta}$. The local boundedness of the function $L$ yields \eqref{limit-alto}.
\end{enumerate}
}
\end{rmk}

\begin{thm}
Under Hypotheses $\ref{hyp-reg-coeff}(ii), (iii)$ and Hypotheses $\ref{hyp-10}$, estimate \eqref{est_neu} is satisfied
and the constant $C_{s,T}$ depends on $\eta_0$, $d$,
$\sup_{(s,T)}L_j$ $(j=1,\dots,7)$, $\sup_{(s,T)}M_1$, $\max_{1\le i,j\le d}\|q_{ij}\|_{C^{0,1}_b((s,T)\times\Rd_{+,\delta})}$.
Further, if the functions $L_j$ $(j=1,\ldots,7)$ are bounded from above in $(s,+\infty)$ and
$q_{ij}\in C^{0,1}_b((s,+\infty)\times\Rd_{+,\delta})$, estimate \eqref{est_neu-neu} holds true, and the constant therein appearing is independent of
$s$ if $L_j$ $(j=1,\ldots,7)$ are bounded from above in $I$ and
$q_{ij}\in C^{0,1}_b(I\times\Rd_{+,\delta})$.
\end{thm}

\begin{proof}
We limit ourselves to proving \eqref{est_neu} and observe that, for any $f\in C^3_c(\Rd_+)$,
the function $v=e^{\Gamma}G_{\mathcal R}(\cdot,s)f$ solves the Cauchy Neumann problem
associated with the operator $\tilde\A$, defined on smooth functions $\zeta$ by
$\tilde{\mathcal A}(t)\zeta= {\rm Tr}(QD^2\zeta)+\langle \tilde b, \nabla_x \zeta\rangle-\tilde c\psi$,
where $\tilde b=b-2 Q \nabla\Gamma$ and $\tilde c=c+({\mathcal A}+c)\Gamma-\langle Q\nabla\Gamma, \nabla\Gamma\rangle$.
Clearly, $G_{\mathcal R}(t,s)f$ satisfies the gradient estimate \eqref{est_neu} if and only if the function $v$ does.
Therefore, in view of Theorem \ref{caso_convesso}, we can limit ourselves to checking that the pair $(\tilde \A,D_d)$ satisfies
Hypotheses \ref{hyp1}(ii)-(iv), \ref{hyp-3} and \ref{hyp-reg-coeff}(ii), (iii). Hypotheses \ref{hyp1}(ii), \ref{hyp1}(iv) and
Hypothesis \ref{hyp-reg-coeff}(ii) are clearly satisfied as well as Hypotheses
\ref{hyp-3} with $\tilde\varphi=e^{\Gamma}\varphi$. Next, we note that due to \eqref{limit-alto}, there exists a positive constant
$M=M_{(s,T)}$ such that $\inf_{I\times \Rd_+}\tilde c\ge M$. Without loss of generality, we can assume that $M_{(s,T)}\ge 0$. Indeed,
as Remark \ref{rem-2.5} shows, we can always reduce to this situation, replacing $v$ by the function $w=e^{M(\cdot-s)}v$.
Finally, let us check Hypotheses \ref{hyp-reg-coeff}(iii).
Estimate \eqref{der_q}(a) is obvious. As far \eqref{der_q}(b) is concerned,
from Hypothesis \ref{hyp-10}(ii), recalling that the support of $\Gamma$ is contained in $\R^{d-1}\times [0,\delta]$ and observing that
$|J_xb|\le L_3+L_4c$, due to \eqref{dissip_c}, and $c\le\tilde c+L_7$, we get
\begin{align*}
|\nabla_x\tilde c|
\le & |\nabla_xc|+d\max_{1\le i,j\le d}\|\nabla_xq_{ij}\|_{C_b((s,T)\times\Rd_{+,\delta})}(\|\nabla\Gamma\|_{\infty}^2+\|D^2\Gamma\|_{\infty})\\
&+d\max_{1\le i,j\le d}\|q_{ij}\|_{C_b((s,T)\times\Rd_{+,\delta})}(\|D^3\Gamma\|_{\infty}
+2\|\nabla\Gamma\|_{\infty}\|D^2\Gamma\|_{\infty})\\
&+|J_xb|\|\nabla\Gamma\|_{\infty}+|b|\|D^2\Gamma\|_{\infty}\\
\le &L_1+L_2c+L_6(L_8d+L_3+L_5+2dL_6L_8)+L_6(L_4+L_5)c\\
\le &L_1+L_6(L_8d+L_3+L_5+2dL_6L_8)+[L_6(L_4+L_5)+L_2]^+L_7\\
&+[L_6(L_4+L_5)+L_2]^+\tilde c,
\end{align*}
in $I\times\Rd_+$, where $L_8=\max_{1\le i,j\le d}\|q_{ij}\|_{C^{0,1}_b((s,T)\times\Rd_{+,\delta})}$.
Similarly, taking Hypothesis \ref{hyp-10}(i) and condition \eqref{baglioni-1} into account we deduce that
\begin{align*}
\langle J_x \tilde b\,\xi, \xi\rangle
\le & (L_3+2dL_6L_8+L_4^+L_7+L_4^+\tilde c)|\xi|^2,
\end{align*}
in $I\times\Rd_{+,\delta}$ and for any $\xi\in\Rd$.
Hence, condition \eqref{dissip_c} is satisfied with $L_4^+$ replacing $L_4$ and a different function $L_3$.
\end{proof}

\section{Examples}
\label{sect_ex}
In this section we provide some class of operators $\A$ which fulfill our assumptions. We confine ourselves to the relevant cases when $\Omega=\Rd_+$ and when $\Omega$ is an exterior domains. In what follows $I$ denotes a right halfline (possibly $I=\R$), $J$ any bounded interval contained in $I$
and $\alpha\in (0,1)$.

\begin{example}
{\rm Let $\A$ and $\B$ be, respectively, the elliptic operator, defined by
\begin{equation}
{\mathcal A}(t)=\omega(t)(1+|x|^2)^r\Delta_x+\langle b(t,x),\nabla_x \rangle-\hat c(t,x)(1+|x|^2)^m,
\label{oper-ex}
\end{equation}
for any $(t,x)\in I\times\Rd_+$, and the operator in \eqref{OpB}.
The coefficient $\omega$ belongs to $C^{\alpha/2}_{\rm loc}(I)$, the entries of the vector $b$ and the function $\hat c$ belong to $C^{\alpha/2,\alpha}_{\rm loc}(I\times\Rd_+)$. Moreover, $\inf_I\omega>0$ and $\inf_{I\times\Rd_+}\hat c=c_0>0$. Finally,  there exist $R>0$, $p\in [0,+\infty)$, such that $(r-1)^+<\max\{p,m\}$, and a function $k_1:I\to\R$ with positive infimum over any $J\subset I$, such that
$\langle b(t,x),x\rangle\le-k_1(t)(1+|x|^2)^{p}|x|^2$ for any
$(t,x)\in I\times\big(\R^d_+\setminus B_R\big)$. As far as the coefficients of the operator $\B$ are concerned, we assume that $\beta_i, \gamma\in C^{(1+\alpha)/2,1+\alpha}_{\rm loc}(\overline{I} \times\R^{d-1})$ $(i=1,\ldots,d)$, $\sum_{i=1}^d\beta_i^2=1$, $\gamma\ge 0$ in $I\times\R^{d-1}$ and $\inf_{(t,x')\in I\times \R^{d-1}}\beta_d(t,x')>0$. Moreover, we assume that $\langle\beta'(t,x'),x'\rangle +\gamma(t,x')(1+|x'|^2)\ge 0$
for any $(t,x')\in I\times\R^{d-1}$.

Under the previous set of assumptions, the function $\varphi$, defined by $\varphi(x)=1+|x|^2$ for any $x\in\overline{\Rd_+}$, satisfies the estimates
\begin{align*}
(\A(t)\varphi)(x)\le 2d\omega_J(1+|x|^2)^r-2k_{1,J}(1+|x|^2)^p|x|^2-\hat c_0(1+|x|^2)^{m+1},
\end{align*}
for any $J\subset I$, $t\in J$ and $x\in\Rd_+\setminus B_R$, where $\omega_J=\sup_{J}\omega$, $k_{1,J}=\inf_Jk_1$.
Moreover, $(\B(t)\varphi)(x',0)= 2\langle \beta'(t,x'),x'\rangle +\gamma(t,x')(1+|x'|^2)$ for any $t\in J$ and any $x'\in\R^{d-1}$.
The assumptions on $m, p, r$ and on $\beta'$ and $\gamma$ show that Hypotheses \ref{hyp-3} hold true. Moreover, the assumptions of
Theorem \ref{thm-comp} are satisfied as well, with $\psi=\varphi$ and $\varepsilon=\max\{p,m\}$. Hence,
the operator $G_{\mathcal B}(t,s)$, associated with the operator $\A$ in \eqref{oper-ex}, is compact for any $(t,s)\in\Lambda$.

In the particular case when $\B=\frac{\partial}{\partial\nu}$ (i.e., Neumann boundary conditions are prescribed), $b(t,x)=-b_0(t)x(1+|x|^2)^p$ for some positive function $b_0\in C^{\alpha/2}_{\rm loc}(I)\cap C_b(I)$, $\hat c\in C^{0,1}_b(I\times \Rd_+)$,
the assumptions of Theorem \ref{caso_convesso} are satisfied with $\eta(t,x)=\omega(t)(1+|x|^2)^r$, $M_1=r$, $L_1\equiv L_3\equiv L_4\equiv 0$, $L_2=m+c_0^{-1}\|\nabla_x \hat c\|_{\infty}$. Hence, the gradient estimate \eqref{est_neu} holds true. If the coefficients $\omega$, $b_i$ ($i=1,\ldots,d$) and $\hat c$ belong to $C_b(I)$ and to $C^{0,1}_b(I\times\Rd_+)$, respectively, then the estimate \eqref{est_neu-neu} holds true as well.}
\end{example}

\begin{example}
\label{exa-Rd}
{\rm
Let
\begin{equation}
\mathcal A(t)=(1+x_d^2)^r{\rm Tr}(\hat Q(t,x)D^2_x)+(1+x_d^2)^p\langle \hat b(t,x),\nabla_x\rangle-\hat c(t,x)(1+x_d^2)^m,
\label{oper-ex-ex}
\end{equation}
for any $(t,x)\in I\times\Rd_+$, where $m,p,r$ are nonnegative constants such that $r<\max\{p+1,m+1\}$. The function $\hat c$ and the entries of $\hat Q$ and $\hat b$
belong to $C^{0,1}_b(J\times\Rd_+)\cap C^{\alpha/2,1+\alpha}_{\rm loc}(I\times\overline{\Rd_+})$ for any $J\subset I$
and $0<c_0(t):=\inf_{\Rd_+}\hat c(t,\cdot)$, the function $c_0$
being locally bounded from below by a positive constant. Moreover,
$\langle \hat Q(t,x)\xi,\xi\rangle \ge \eta_0|\xi|^2$ for any $(t,x)\in I\times\Rd_+$, $\langle J_x\hat b(t,x)\xi,\xi\rangle\le -\sigma_0(t)|\xi|^2$, $\langle \hat b(t,x),x\rangle\le -\sigma_0(t)|x|^2$ for any $(t,x)\in I\times\Rd_+$, any $J \subset I$, $\xi\in\Rd$ and some continuous function $\sigma_0:I\to (0,+\infty)$. Finally,
the coefficients of the operator $\B$ in \eqref{OpB} satisfy
$\beta\in C^{2+\alpha}_{\rm loc}(\R^{d-1})$, $\gamma\in C^{1+\alpha}_{\rm loc}(\R^{d-1})$, $\sum_{i=1}^d\beta_i^2\equiv 1$, $\gamma\ge 0$ in $\R^{d-1}$, $\inf_{\R^{d-1}}\beta_d>0$ and $\langle\beta'(x'),x'\rangle+\gamma(x')(k^2+|x'|^2)\ge 0$ for any $x'\in\R^{d-1}$ and some positive constant $k\ge 1$.

Under this set of assumptions, the function $\varphi$, defined by $\varphi(x)=\sqrt{k^2+|x|^2}$ for any $x\in\Rd_+$, satisfies Hypothesis \ref{hyp-3}. Indeed,
\begin{align*}
\frac{(\A(t)\varphi)(x)}{\varphi(x)}
\le \sqrt{d}\|\hat Q\|_{C_b(J\times\Rd_+;\R^{d^2})}(1+x_d^2)^{r-1}\!\!-\!\frac{\sigma_0(t)|x|^2}{k^2+|x|^2}(1+x_d^2)^p\!\!-\!\!c_0(t)(1+x_d^2)^m\!,
\end{align*}
for any $(t,x)\in J\times\Rd_+$ and any $J\subset I$. The choice of $m,p,r$ shows that the right-hand side of the previous inequality tends to $-\infty$ as $x_d\to +\infty$, uniformly with respect to $t\in J$. Hence, $(\A\varphi)/\varphi$ is bounded in $J\times\Rd_+$.
Since, clearly, $\B\varphi\ge 0$ on $\partial \Rd_+$, the function $\varphi$ satisfies Hypotheses \ref{hyp-3}.
Similarly, Hypothesis \ref{hyp-reg-coeff}(iii) is satisfied with $\eta(t,x)=\eta_0(1+x_d^2)^r$, $M_1(t)=\eta_0^{-1}(r\|\hat q_{ij}(t,\cdot)\|_{\infty}+\|\nabla_x \hat q_{ij}(t,\cdot)\|_{\infty})$,
$L_1\equiv L_4\equiv 0$, $L_2(t)=m+(c_0(t))^{-1}\|\nabla_x \hat c(t,\cdot)\|_{\infty}$, $L_3(t)=-\sigma_0(t)+2p\|\hat b(t,\cdot)\|_{\infty}$ if $p \le 1/2$ and $L_3(t)=\max\{(2p-1)^{2p-1}\|\hat b(t,\cdot)\|_\infty^{2p}(\sigma_0(t))^{1-2p},-\sigma_0(t)+2p\|\hat b(t,\cdot)\|_{\infty}\}$, otherwise. Hence, the gradient estimate \eqref{est_neu} holds true.
If $c_0$ and $\sigma_0$ are bounded from below in $I$ by a positive constant, $\hat b_j, D_i\hat c\in C_b(I\times\Rd_+)$ and $q_{ij}\in C^{0,1}_b(I\times\Rd_+)$ ($i,j=1,\ldots,d$), then estimate \eqref{est_neu-neu} holds true as well.}
\end{example}

We now consider the case when $\Omega\subset \Rd$ is an exterior domain.

\begin{example}
\label{example-ext-1}
{\rm  Assume that $\Omega$ has a boundary uniformly of class $C^{2+\alpha}$.
Let ${\mathcal A}$, ${\mathcal B}$ be the operators in \eqref{operator} and \eqref{OpB}. Assume that Hypotheses \ref{hyp1}(ii)-(iv) and Hypotheses \ref{hyp1-bc} are satisfied. Assume that
\begin{equation}
\ell_J:=\sup_{(t,x)\in J\times (\Rd_+\setminus B_1)}\frac{{\rm Tr}(Q(t,x))+\langle b(t,x),x\rangle}{|x|^2}<+\infty,\qquad\;\,t\in J,
\label{carboni}
\end{equation}
for any $J\subset I$. For instance, condition \eqref{carboni} is satisfied when $Q$ is bounded in $\Om_I$ and $\langle b(t,x),x\rangle$ grows at infinity at most quadratically, uniformly with respect to $t\in J$ for any bounded interval $J \subset I$.

If $\gamma_0:=\inf_{\Omega_I}\gamma\ge 0$, under the previous assumptions, the function $\varphi:\R^d\to\R$, defined by $\varphi(x)=(1-r_{\Om}(x))\vartheta(x)+(1-\vartheta(x))(1+|x|^2)$ for any $x\in\R^d$, satisfies Hypotheses \ref{hyp-3}. Here, $\vartheta$ is
any smooth function with ${\rm supp}(\vartheta)\subset\Omega_{\delta}$ and $\vartheta\equiv 1$ in $\Omega_{\delta/2}$, where
$\delta$ is defined in Remark \ref{local charts}(b). If, further,
$c(t,x)= \hat{c}(t,x)(1+|x|^2)^m$ for any $(t,x)\in\Omega_I$, some $m>1$ and some smooth enough and bounded function $\hat c$ with positive infimum on $\Om_I$, then the function $\psi:\Om\to \R$ defined by $\psi(x)=1+|x|^2$, $x \in \Om$, satisfies the assumption of Theorem \ref{thm-comp} with $\varepsilon=m$. Hence,
the evolution operator $G_{\mathcal B}(t,s)$ is compact.

Let us now assume that $\gamma_0<0$.
Fix a function $\zeta\in C^{2+\alpha}_b([0,+\infty))$ with positive infimum, such that $\zeta(0)=1$, $\zeta\equiv 1/2$ in $[\delta/2,+\infty)$ and $\zeta'(0)<0$.
Further, let $\sigma$ be a constant greater than $\max\left\{1,\frac{\gamma_0}{\beta_0\zeta'(0)}\right\}$,
where $\delta$ is as above.
Then, the function $\Phi$, defined by $\Phi(x)=\zeta(\sigma r_{\Omega}(x))$ for any $x\in\Om$, belongs to $C^{2+\alpha}_b(\Om)$
and has positive infimum. Moreover,
$({\mathcal B(t)}\Phi)(x)\ge -\sigma\zeta'(0)\beta_0+\gamma_0> 0$ in $I\times\partial\Omega$, due to the choice of $\sigma$.
Moreover, since $\Phi$ is constant outside a neighborhood of $\partial\Omega$, $(\A\Phi)/\Phi$ is bounded in $J\times\Omega$ for any
$J\subset I$, i.e., Hypothesis \ref{ipofisi} is satisfied.

Finally, analogous computations as above show that the function $\varphi$, defined by $\varphi(x)= \zeta(\sigma r_\Om(x))+ (1-\zeta(\sigma r_{\Omega}(x))|x|^2$ for any $x\in\Omega$, satisfies Hypotheses \ref{hyp-3}. Therefore, the results in Theorem \ref{existence-bis} can be applied.}
\end{example}

\begin{example}
{\rm
Assume that $\partial \Om$ is uniformly of class $C^{3+\alpha}$. Let $\A$ and $\B$ be as in Example \ref{exa-Rd} with $x_d^2$ being replaced
by $|x|^2$. The two functions  $\varphi$, introduced in Example \ref{example-ext-1}, satisfy Hypotheses \ref{hyp-3}. Moreover, arguing as for the operator $\A$ in \eqref{oper-ex-ex}, it can be easily shown that Hypotheses \ref{hyp-reg-coeff} are satisfied.}
\end{example}

\appendix
\section{Technical results}
\label{sect_app}

\begin{thm}\label{appendicite}
Let $\Omega$ be an unbounded domain with a boundary uniformly of class $C^{2+\alpha}$ for some $\alpha \in (0,1)$. Let $\mathcal{A}$
be the uniformly nonautonomous elliptic operator defined by \eqref{operator}, with coefficients in $C^{\alpha}([a,b], C_b(\overline{\Omega}))$
and let ${\mathcal B}=I$ or ${\mathcal B}=\langle \beta,\nabla\rangle+\gamma$ where $\beta_i, \gamma\in C^{\sigma}([a,b];C_b^1(\overline\Omega))$ for
some $\sigma>1/2$ and any $(i=1,\ldots,d)$. Then,
\begin{equation}
(G_{\mathcal B}(t,s_2)f)(x)-(G_{\mathcal B}(t,s_1)f)(x)= -\int_{s_1}^{s_2}(G_{\mathcal B}(t,r)\mathcal{A}(r)f)(x)\, dr,
\label{bietola_appen}
\end{equation}
for any $s_1, s_2, t\in [a,b]$, with $t\ge\max\{s_1,s_2\}$, any $x \in \Om$ and any $f\in C^2_c(\Omega)$.
\end{thm}
\begin{proof}
Estimate \eqref{bietola_appen} has been proved in \cite[Thms. 2.3(ix) \& 6.3]{Acq88Evo} when $\Omega$ is bounded, but the arguments used in
\cite{Acq88Evo} can be extended to our situation. The two cases being similar, we limit ourselves to dealing with the boundary operator
$\langle \beta,\nabla\rangle+\gamma I$.

For any $t \in[a,b]$ we denote by $A(t)$ the realization of $\mathcal{A}(t)$ in $C_b(\overline{\Om})$ with domain
$D(A(t))=\left\{u \in \bigcap_{p\ge 1}W^{2,p}_{\rm{loc}}(\Omega)\cap C_b(\overline{\Omega}): \,\,
\mathcal{A}(t)u \in C_b(\overline \Om),\, \mathcal{B}(t)u=0~{\rm on}~\partial\Omega\right\}$. To check the assumptions in \cite[Thm. 2.3(ix)]{Acq88Evo},
we have to prove the following properties:
\begin{enumerate}[\rm (i)]
\item
there exists $\omega \in \R$ such that, for any $t \in [a,b]$ and some
$\theta \in (\frac{\pi}{2}, \pi)$, $\rho(A(t))\supset \omega+\Sigma_{\theta}$, where $\Sigma_\theta=
\{\lambda \in \mathbb C\setminus\{0\}: |\rm{arg}\lambda|\le \theta\}$,
and the resolvent estimate in $\omega+\Sigma_{\theta}$ is uniform with respect to $t \in[a,b]$;
\item
there exist a positive constant $C$ and $0\le\theta_i<\alpha_i\le 2$ ($i=1,2$) such that
\begin{align}
&\|(A(t)-(\omega+1)I)R(\lambda+\omega+1, A(t))(R(\omega+1,A(s))-R(\omega+1,A(t)))\|_{\mathcal{L}(C_b(\Om))}\nnm\\
\le &C\Big (|\lambda|^{\theta_1-1}|t-s|^{\alpha_1}+|\lambda|^{\theta_2-1}|t-s|^{\alpha_2}\Big )
\label{latticini}
\end{align}
for any $t,s \in [a,b]$, $\lambda \in \Sigma_{\theta}$.
\end{enumerate}
Property (i) follows from the
estimate proved by H.B. Stewart (\cite{Ste80Gen}) in the autonomous case, noting that
the constants appearing in the proof depend only on the ellipticity constant, the modulus of continuity and the $L^{\infty}$-norm of the coefficients.
Such estimate shows that
\begin{align}
&|\lambda|\|u\|_{\infty}+|\lambda|^{\frac{1}{2}}\|\nabla u\|_{\infty}+|\lambda|^{\frac{d}{2p}}\sup_{x_0 \in \overline{\Om}}
\|D^2u\|_{L^p(\Om \cap B_{1/\sqrt{|\lambda|}}(x_0))}\nnm\\
&\le M\Big(|\lambda|^{\frac{d}{2p}}\inf_{t\in [a,b]}\sup_{x_0 \in \overline{\Om}}\|\lambda u-\mathcal{A}(t)u\|_{L^{p}(\Om
\cap B_{1/\sqrt{|\lambda|}}(x_0))}+|\lambda|^{\frac{1}{2}}\|g\|_{\infty}\nnm\\
&\qquad\quad+|\lambda|^{\frac{d}{2p}}\sup_{x_0 \in \overline{\Om}}\|\nabla g\|_{L^p(\Om \cap B_{1/\sqrt{|\lambda|}}(x_0))}\Big),
\label{apriori}
\end{align}
for any $\lambda$, with ${\rm Re}\lambda\ge\omega$, some $\omega>0$, any function $u\in W^{2,p}_{\rm loc}(\Omega)\cap C^1_b(\overline\Omega)$,
some $p>\max\{d/(2\alpha),d\}$. Here, $g$ is any
$W^{1,p}_{\rm{loc}}(\Om)$-extension of ${\mathcal B}(t)u$, and $M$ is a positive constant
independent of $\lambda$, $u$ and $g$. Since any operator $A(t)$ is sectorial, its resolvent set contains a right-halfline. Estimate \eqref{apriori}
shows
that $R(\cdot,A(t))$ is bounded in $\rho(A(t))\cap \{\lambda\in\C : {\rm Re}\lambda\ge\omega\}$ and this implies that
$\rho(A(t))\supset \{\lambda\in\C : {\rm Re}\lambda\ge\omega\}$
and
$\|R(\lambda,A(t))\|_{{\mathcal L}(C_b(\overline\Omega))}\le M|\lambda|^{-1}$ for any $\lambda\in\C$ with ${\rm Re}\lambda\ge\omega$.
Indeed, the norm of $R(\lambda,A(t))$ blows up as $\lambda$ approaches the boundary of $\rho(A(t))$. (see e.g. \cite[Prop. A.0.3]{Lun95Ana}).
Moreover, a simple argument based on von Neumann series and the previous estimate (see e.g., \cite[Prop. 3.1.11]{Lun95Ana}) shows that
$\rho(A(t))$ contains
the sector $\omega+\Sigma_{\theta}$ for $\theta=\pi-\arctan(2M)\in (\frac{\pi}{2},\pi)$ and
$\|R(\lambda,A(t))\|_{{\mathcal L}(C_b(\overline\Omega))}\le 2M|\lambda-\omega|^{-1}$ for any $\lambda\in\omega+\Sigma_{\theta}$.

Property (ii) can be proved arguing as in \cite[Thm. 6.3]{Acq88Evo}. For the reader's convenience we enter into details and we prove it with
$\alpha_1=\alpha$ $\alpha_2=\sigma$, $\theta_1=d/(2p)$ and $\theta_2=1/2$ (note that our assumptions on $\sigma$ and $p$ guarantee that the
conditions $\theta_1<\alpha_1$ and $\theta_2<\alpha_2$ are satisfied).
Fix $f\in C_b(\overline\Omega)$, $\lambda\in\C$ with positive real part, and let $v=R(\mu,A(s))f$ and
$u=R(\lambda+\mu,A(t))(\lambda+\mu-A(s))R(\mu,A(s))f$, where $\mu=\omega+1$.
Clearly, $u-v=(A(t)-\mu I)R(\lambda+\mu, A(t))(R(\mu,A(s))-R(\mu,A(t)))f$. So, if we set $w_{\lambda,\mu}=u-v$, estimate \eqref{latticini} becomes
\begin{align}
\|w_{\lambda,\mu}\|_{\infty}
\le C\Big (|\lambda|^{\frac{d}{2p}-1}|t-s|^{\alpha}+|\lambda|^{-\frac{1}{2}}|t-s|^{\sigma}\Big )\|f\|_{\infty},
\label{A-2}
\end{align}
for some constant $C$, independent of $f,\lambda,\alpha,\sigma,t,s$.

Applying estimate \eqref{apriori} to the function $w_{\lambda,\mu}\in C^1_b(\overline\Omega)\cap\bigcap_{p<+\infty}W^{2,p}_{\rm loc}(\Omega)$
which satisfies the elliptic problem
\begin{eqnarray*}
\left\{
\begin{array}{ll}
(\lambda+\mu)w_{\lambda,\mu}-{\mathcal A}(t)w_{\lambda,\mu}=({\mathcal A}(t)-{\mathcal A}(s))v, & {\rm in}\;\Om,\\[1mm]
{\mathcal B}(t)w_{\lambda,\mu}=[{\mathcal B}(s)-{\mathcal B}(t)]v, & {\rm on}\;\partial\Omega,
\end{array}.
\right.
\end{eqnarray*}
we get
\begin{align}
\|w_{\lambda,\mu}\|_{\infty}
\le M_1\bigg (&|\lambda+\mu|^{\frac{d}{2p}-1}\sup_{x_0 \in \overline{\Om}}\|({\mathcal A}(t)
-{\mathcal A}(s))v\|_{L^p(\Om \cap B_{1/\sqrt{|\lambda+\mu|}}(x_0))}\nnm\\
&+|\lambda+\mu|^{\frac{d}{2p}-1}\sup_{x_0 \in \overline{\Om}}
\|\nabla_x({\mathcal B}(s)-\nabla_x{\mathcal B}(t))v\|_{L^p(\Om \cap B_{1/\sqrt{|\lambda+\mu|}}(x_0))}\nnm\\
&+|\lambda+\mu|^{-\frac{1}{2}}\|({\mathcal B}(s)-{\mathcal B}(t))v\|_{\infty}\bigg ),
\label{A-0}
\end{align}
for some positive constant $M_1$, independent of $f,\lambda,\mu,t,s,$. From now on, we denote by $L_j$ positive constants independent of
$f,\lambda,\mu,\alpha,\sigma,t,s,x_0$. The smoothness of the coefficients imply that
\begin{align}
& \|({\mathcal A}(t)-{\mathcal A}(s))v\|_{L^p(\Om \cap B_{1/\sqrt{|\lambda+\mu|}}(x_0))}\nnm\\
\le & L_1|t-s|^{\alpha}\left (\|D^2v\|_{L^p(\Om \cap B_{1/\sqrt{|\lambda+\mu|}}(x_0))}+
|\lambda+\mu|^{-\frac{d}{2p}}\|v\|_{C^1_b(\overline\Omega)}\right ),
\label{A}\\[3mm]
&\|({\mathcal B}(s)-{\mathcal B}(t))v\|_{\infty}\le L_2|t-s|^{\sigma}\|v\|_{C^1_b(\overline\Omega)}
\label{B}
\\[3mm]
&\|\nabla_x({\mathcal B}(s)v-{\mathcal B}(t))v\|_{L^p(\Om \cap B_{1/\sqrt{|\lambda+\mu|}}(x_0))}\nnm\\
\le &L_3|t-s|^{\sigma}\left (\|D^2v\|_{L^p(\Om \cap B_{1/\sqrt{|\lambda+\mu|}}(x_0))}+
|\lambda+\mu|^{-\frac{d}{2p}}\|v\|_{C^1_b(\overline\Omega)}\right ).
\label{C}
\end{align}
Moreover, estimate \eqref{apriori} applied to $v=R(\mu,A(s))f$ shows that
\begin{equation}
\|v\|_{C^1_b(\overline\Omega)}+\sup_{x_0 \in \overline{\Om}}\|D^2v\|_{L^p(\Om \cap B_{1/\sqrt{|\lambda+\mu|}}(x_0))}\le L_4\|f\|_{\infty},
\label{B-G:1-7}
\end{equation}
where we take into account that the ball $B_{1/\sqrt{|\lambda+\mu|}}(x_0)$ is contained into the ball $B_{1/\sqrt{|\mu|}}(x_0)$ since
${\rm Re}\lambda>0$ and $\mu>0$. Replacing \eqref{A}-\eqref{B-G:1-7} into \eqref{A-0}, taking into account that $|\lambda+\mu|\ge \mu>1$
(which implies that $|\lambda+\mu|^{-1}\le |\lambda+\mu|^{\frac{d}{2p}-1}\le |\lambda+\mu|^{-\frac{1}{2}}$), that $|\lambda+\mu|\ge |\lambda|$
(since ${\rm Re}\lambda>0$) and our choice of $p$, we deduce estimate \eqref{A-2} in the halfplane $\{\lambda\in\C: {\rm Re}\lambda\ge 0\}$.

To extend \eqref{A-2} to $\Sigma_{\theta}$, we use again the proof of \cite[Prop. 3.1.11]{Lun95Ana} which shows that 
$R(\lambda+\mu,A(t))=\sum_{n=0}^{+\infty}(-{\rm Re} \lambda)^nR(\mu+i{\rm Im}\lambda,A(t))^{n+1}$
for any $\lambda\in\Sigma_{\theta}$ with negative real part.
Therefore,
\begin{align*}
\|w_{\lambda,\mu}\|_{\infty}
=&\left\|\sum_{n=0}^{+\infty}(-{\rm Re} \lambda)^nR(\mu+i{\rm Im}\lambda,A(t))^nw_{\mu+i{\rm Im}\lambda}\right\|_{\infty}\\
\le &
L_5\Big (|{\rm Im}\lambda|^{\frac{d}{2p}-1}|t-s|^{\alpha}+|{\rm Im}\lambda|^{-\frac{1}{2}}|t-s|^{\sigma}\Big )\\
&\quad\times\sum_{n=0}^{+\infty}|{\rm Re} \lambda|^n\|R(\mu+i{\rm Im}\lambda,A(t))\|^n_{{\mathcal L}(C_b(\overline\Omega))}.
\end{align*}

To estimate the series, we recall that the choice of $\theta$ implies that $|{\rm Re}\lambda|\le (2M)^{-1}|{\rm Im}\lambda|$
for any $\lambda\in\Sigma_{\theta}$. This and the resolvent estimate proved above show that
$\sum_{n=0}^{+\infty}|{\rm Re} \lambda|^n\|R(\mu+i{\rm Im}\lambda,A(t))\|^n_{{\mathcal L}(C_b(\overline\Omega))}\le 2$ for
any $\lambda\in\Sigma_{\theta}$, and estimate \eqref{A-2} follows in the whole of $\Sigma_{\theta}$.
\end{proof}

\begin{lemm}\label{geometric}
Assume that Hypothesis \ref{hyp-reg-coeff}(i) holds and that $\beta \in C^{2+\alpha}_{\rm loc}(\overline{\Omega},\Rd)$ is bounded together with all its derivatives on $\partial\Omega$.
Then, there exists $r_0>0$ such that, for any $x_0\in\partial\Omega$, there exists $\phi_{x_0}\in C^2(B_{r_0}(x_0),\Rd)$ such that
\begin{equation}
J\phi_{x_0}(x)\beta(x)= \rho_{x_0}(x)e_d,\qquad\;\, x \in B_{r_0}(x_0)\cap \partial \Om,\;\,x_0\in \partial\Omega,
 \label{costarica}
\end{equation}
where $e_d=(0,\ldots,0,1)^T$, for some continuous function $\rho_{x_0}$, which nowhere vanishes on $B_{r_0}(x_0)\cap\partial\Omega$.
Moreover, $\phi_{x_0}(B_{r_0}(x_0)\cap\Omega)$ is a bounded domain contained in $\Rd_+$, $\phi_{x_0}(x)\in\overline{\Rd_+}$ if
and only if $x\in B_{r_0}(x_0)\cap\overline\Omega$,
$\phi_{x_0}(B_{r_0}(x_0)\cap\partial\Omega)\subset\overline B_1^+\cap \partial\Rd_+$ and there exists a positive constant $\Phi$ such that
\begin{equation}
\sup_{x_0 \in \partial\Omega}\left(\|\phi_{x_0}\|_{C^{2+\alpha}(B_{r_0}(x_0))}+\|\phi_{x_0}^{-1}\|_{C^{2+\alpha}(\phi_{x_0}(B_{r_0}(x_0)))}
\right)\leq \Phi.
\label{Fi}
\end{equation}
\end{lemm}
\begin{proof}
Let $(B_R(x_h), \psi_h)_{h \in \N}$ be the covering of $\partial\Omega$ in Remark \ref{local charts}, with $x_h \in \partial \Om$ and $R>0$.
For any $x=(x',x_d)\in \overline{B_1}=\psi_h(\overline{B_R(x_h)})$, we introduce the vector
$\Upsilon_h(x)=J \psi_h(\psi_h^{-1}(x))\beta(\psi_h^{-1}(x',0))$.

 From now on, we fix an arbitrary $h\in\N$.
We claim that the last component $\Upsilon_h^d$ of $\Upsilon_h$ nowhere vanishes on $\overline{B_1^+}$ or, equivalently, on
$\overline{B_1^+}\cap\partial\Rd_+$. As it is easily seen,
\begin{equation}
\Upsilon_h^d(x)=\langle\nabla_x\psi_h^d(\psi_h^{-1}(x)),\beta(\psi_h^{-1}(x))\rangle,\qquad\;\,x\in \overline{B_1^+}\cap\partial\Rd_+.
\label{lemma-geom}
\end{equation}
By Remark \ref{local charts}(c),
$\nabla\psi_h^d(\psi_h^{-1}(x))=-|\nabla_x\psi_h^d(\psi_h^{-1}(x))|\nu(\psi_h^{-1}(x))$ for any $x\in \overline{B_1^+}\cap\partial\Rd_+$.
Formula \eqref{lemma-geom} now shows that
\begin{equation*}
|\Upsilon_h^d(x)|=|\nabla_x\psi_h^d(\psi_h^{-1}(x))|\langle\beta(\psi_h^{-1}(x)),\nu(\psi_h^{-1}(x))\rangle,\qquad\;\,x\in\overline{B_1^+}
\cap\partial\Rd_+.
\end{equation*}
 Recalling that $\beta$ satisfies Hypothesis \ref{hyp1-bc}(iii) and the gradient of $\psi_h$ nowhere vanishes in $B_R(x_h)$, we conclude that
$\Upsilon_h^d$ nowhere vanishes on $\overline{B_1^+}\cap\partial\Rd_+$.

We can thus define the function $\tilde\phi_h: \overline{B_R(x_h)}\to\Rd$ by setting
\begin{align*}
\tilde\phi_h(x)=\left(\psi_h^1(x)-\frac{\Upsilon_h^1(\psi_h(x))}{\Upsilon_h^d(\psi_h(x))}\psi_h^d(x), \dots,\psi_h^{d-1}(x)
-\frac{\Upsilon_h^{d-1}(\psi_h(x))}{\Upsilon_h^d(\psi_h(x))}\psi_h^{d}(x), \psi_h^d(x)\right),
\end{align*}
for any $x\in \overline {B_R(x_h)}$. It is easy to notice that $\tilde\phi_h(B_R(x_h) \cap \partial \Om)=\overline{B_1^+}\cap\partial\Rd_+$,
that $\tilde\phi_h(B_R(x_h) \cap \Om)$ is a bounded subset of $\Rd_+$ and
$\tilde\phi_h(x)\in\overline{\Rd_+}$ if and only if $x\in \overline{B_R(x_h)}\cap
\overline{\Omega}$. Indeed $\psi_h$ and $\tilde\phi_h$ agree on $B_R(x_h)\cap\partial\Omega$, $\tilde\phi_h^d\equiv \psi_h^d$ in $B_R(x)\cap\overline\Omega$
and $\psi_h^d(x)>0$ (resp. $\psi_h^d(x)=0$) if and only if $x\in B_R(x_h)\cap\Omega$
(resp. $x\in B_R(x_h)\cap\partial\Omega$).
Moreover, $\tilde\phi_h\in C^{2+\alpha}(B_R(x_h)\cap\Omega)$ and $\sup_{h\in\N}\|\tilde\phi_h\|_{C^{2+\alpha}(B_R(x_h)\cap\Omega)}<+\infty$.

Let us now prove that
\begin{equation}
J\tilde\phi_h(x)\beta(x)= -\langle\beta(x),\nu(x)\rangle|\nabla \psi_h^d(x)|e_d,\qquad\;\, x \in \partial \Om\cap B_R(x_h).
\label{costarica-0}
\end{equation}
For this purpose, let us fix $x \in \partial \Om\cap B_R(x_h)$. Since
$\langle \beta(x), \nabla \psi^k_h(x)\rangle= \Upsilon^k_h(\psi_h(x))$ for any $k=1,\ldots,d$, it holds that
\begin{align*}
&(J\tilde\phi_h(x)\beta(x))_k=\langle \beta(x), \nabla \psi_h^k(x)\rangle-\frac{\Upsilon_h^k(\psi_h(x))}{\Upsilon_h^d(\psi_h(x))}
\langle \beta(x), \nabla \psi_h^d(x)\rangle=0,\qquad k\le d-1,\\
&
(J\tilde\phi_h(x)\beta(x))_d=\langle \beta(x),\nabla \psi_h^d(x)\rangle=-\langle\beta(x),\nu(x)\rangle|\nabla \psi_h^d(x)|.
\end{align*}

Let us now recall that $\bigcup_{h\in\N}B_{R/2}(x_h)\supset \Omega_{\varepsilon}$ for some $\varepsilon$.
We now want to prove that there exists $r_0>0$ such that, for any $h\in\N$ and any $x_0\in\partial\Omega\cap B_{R/2}(x_h)$,
the function $\tilde\phi_h$ is invertible in $B_{r_0}(x_0)$. For this purpose, we observe that
\begin{equation*}
J \tilde\phi_h(x)=J\psi_h(x)-
\left(\begin{array}{ccc}
\frac{\Upsilon_h^1(\psi_h(x))}{\Upsilon_h^d(\psi_h(x))}D_1\psi_h^d(x)&\dots& \frac{\Upsilon_h^1(\psi_h(x))}{\Upsilon_h^d(\psi_h(x))}D_d\psi_h^d(x)\\
\vdots & \vdots& \vdots\\
\frac{\Upsilon_h^{d-1}(\psi_h(x))}{\Upsilon_h^d(\psi_h(x))}D_1\psi_h^d(x)& \dots &
\frac{\Upsilon_h^{d-1}(\psi_h(x))}{\Upsilon_h^d(\psi_h(x))}D_d\psi_h^d(x)\\
0 & \dots & 0
\end{array}
\right),
\end{equation*}
for any $x \in B_R(x_h)\cap \partial \Om$. Hence, ${\rm det}(J\tilde\phi_h(x))={\rm det}(J\psi_h(x))\neq 0$ for such $x$'s.
Now, Remark \ref{local charts} shows that there exists a positive constant $C$, independent of $h$, such that
${\rm det}(J\tilde\phi_h(x))\ge C$ for any $x\in B_R(x_h)\cap\partial\Omega$. The inverse mapping theorem and, again 
the equiboundedness of the norms the functions $\psi_h$ and $\psi_h^{-1}$ show that we can determine $r_0\in (0,R/2)$, independent of $h$, such that $\tilde\phi_h$ is invertible in
$B_{r_0}(x_0)$ for any $x_0\in \partial\Omega\cap B_{R/2}(x_h)$ and its inverse map belongs to $C^{2+\alpha}(\phi_h(B_{r_0}(x_0)\cap\Omega))$
with $C^{2+\alpha}$-norm bounded uniformly with respect to $h$ and $x_0$. We set $\phi_{x_0,h}=(\tilde\phi_h)_{|B_{r_0}(x_0)\cap\Omega}$.

For any $x_0\in\partial\Omega$, we denote by $h(x_0)$ the smallest integer such that $x_0\in B_{R/2}(x_h)$, and we define
$\phi_{x_0}=\phi_{x_0,h(x_0)}$. From the previous results, we know that the family $\{\phi_{x_0}: x_0\in\partial\Omega\}$ satisfies \eqref{Fi}.
Moreover, formula \eqref{costarica-0} yields \eqref{costarica} with
$\rho_{x_0}(x)=-|\nabla\psi_{h(x_0)}^d(x)|\langle\beta(x),\nu(x)\rangle$.
\end{proof}

\begin{lemm}
\label{lemma-approx}
For any $0<r_1<r_2<r_0$ and any $x_0\in\partial\Omega$, there exist a function $\vartheta\in C^{\infty}_c(\overline{\Rd_+})$
and a positive constant $C$ such that $\chi_{\phi_{x_0}(B_{r_1}(x_0)\cap\Omega)}\le \vartheta\le \chi_{\phi_{x_0}(B_{r_2}(x_0)\cap\overline\Omega)}$,
$D_d\vartheta\equiv 0$ on $\partial\Rd_+$ and 
$\|D^k\vartheta\|_{\infty}\le K(r_2-r_1)^{-k}$ for $k=0,1,2,3$ and some positive constant $K$, where $\phi_{x_0}$ and $r_0$ are as in Lemma $\ref{geometric}$.
\end{lemm}

\begin{proof}
We fix $r_1, r_2$ as in the statement and let $\varepsilon=\frac{r_2-r_1}{6\Phi}$, where $\Phi$ is defined in \eqref{Fi}, and let
$\vartheta_0$ be the convolution of a standard mollifier, supported in the ball $B_{\varepsilon}$, and
the characteristic function of the set $\{x\in\Rd: d(x,E)\le\varepsilon\}$, where
$E=\{x\in\Rd: (x',|x_d|)\in\phi_{x_0}(B_{r_1}(x_0)\cap\Omega)\}$.
The function $\vartheta_0$ is smooth and even with respect to the last coordinate, hence $D_d\vartheta_0\equiv 0$ on $\partial\Rd_+$.
We set $\vartheta=(\vartheta_0)_{|\overline{\Rd_+}}$. Clearly, $\vartheta\equiv 1$ in
$\phi_{x_0}(B_{r_1}(x_0)\cap\Omega)$ and its support is contained in the set
$F=\{x\in\overline{\Rd_+}: d(x,\phi_{x_0}(B_{r_1}(x_0)\cap\Omega))\le 2\varepsilon\}$.
We claim that $F\subset\phi_{x_0}(\overline{B_{\frac{r_1+r_2}{2}}(x_0)}\cap\overline\Omega)$.
Note that it is enough to prove that $F\subset\phi_{x_0}(\overline{B_{\frac{r_1+r_2}{2}}(x_0)})$. Indeed, from the definition of the
function $\phi_{x_0}$ it follows that
$\phi_{x_0}(\overline{B_{\frac{r_1+r_2}{2}}(x_0)})\cap\overline{\Rd_+}=\phi_{x_0}(\overline{B_{\frac{r_1+r_2}{2}}(x_0)}\cap\overline\Omega)$.
Since
$|x-x'|=|\phi_{x_0}^{-1}(\phi_{x_0}(x))-\phi_{x_0}^{-1}(\phi_{x_0}(x'))|\le\Phi |\phi_{x_0}(x)-\phi_{x_0}(x')|$
for any $x\in\partial B_{r_1}(x_0)$ and $x'\in \partial B_{\frac{r_1+r_2}{2}}(x_0)$, we immediately deduce that
\begin{align*}
d(\phi_{x_0}(\partial B_{r_1}(x_0)),\phi_{x_0}(\partial B_{\frac{r_1+r_2}{2}}(x_0)))
\ge {r_2-r_1\over 2\Phi}=3\varepsilon.
\end{align*}
Noting that
\begin{align*}
d(\phi_{x_0}(B_{r_1}(x_0)),\Rd\setminus\phi(B_{\frac{r_1+r_2}{2}}(x_0)))=&d(\partial \phi_{x_0}(B_{r_1}(x_0),\partial\phi_{x_0}(B_{\frac{r_1+r_2}{2}}(x_0)))\\
=&d(\phi_{x_0}(\partial B_{r_1}(x_0)),\phi_{x_0}(\partial B_{\frac{r_1+r_2}{2}}(x_0))),
\end{align*}
we conclude that $d(x,\phi_{x_0}(B_{r_1}(x_0)))\ge 3\varepsilon$ for any $x\in\Rd\setminus\phi(B_{\frac{r_1+r_2}{2}}(x_0))$ and this shows that
$F\subset\phi_{x_0}(\overline{B_{\frac{r_1+r_2}{2}}(x_0)})$ as claimed.
Indeed, if $x\in F$, it holds that $d(x,\phi_{x_0}(B_{r_1}(x_0))\le d(x,\phi_{x_0}(B_{r_1}(x_0)\cap\Omega))\le 2\varepsilon$.
It thus follows that
${\rm supp}\,\vartheta\subset\phi_{x_0}(B_{r_2}(x_0)\cap\overline\Om)$.

Finally, we observe that $\|D^k\vartheta\|_{\infty}\le M\varepsilon^{-k}$ for $k=0,1,2,3$ and some positive constant $M$. Our choice
of $\varepsilon$ leads to the estimates in the statement.
\end{proof}

\begin{lemm}
\label{lem-int}
Let $f:[a,b]\times\Omega\to\R$ be a bounded and continuous function and let $g:\Omega\to\R$ be the function
defined by
$g(x)=\int_a^bf(t,x)dt$ for any $x\in\Omega$.
Then, $g\in C_b(\Omega)$. Moreover, for any bounded linear operator $T:C_b(\Omega)\to C_b(\Omega)$, which transforms
bounded sequence of continuous functions, converging locally uniformly in $\Omega$, into sequences with the same properties,
it holds that
$(Tg)(x)=\int_a^b(Tf(t,\cdot))(x)dt$ for any $x\in\Omega$.
\end{lemm}

\begin{proof}
Showing that $g\in C_b(\Omega)$ is an easy task left to the reader.

To prove the last part of the proof, for any $n\in\N$, let $g_n=\sum_{k=0}^{n-1}f(t_k,\cdot)(t_{k+1}-t_k)$,
where $t_k=a+k(b-a)/n$ for $k\in\{0,\ldots,n\}$. Clearly, $g_n$ converges to $g$ locally uniformly in $\Omega$. Further, $\|g_n\|_{\infty}\le \|f\|_{\infty}(b-a)$ for any $n\in\N$. Hence,
$Tg_n$ converges to $Tg$ locally uniformly in
$\Omega$ as $n\to +\infty$. As is immediately seen,
$(Tg_n)(x)=\sum_{k=0}^{n-1}(Tf(t_k,\cdot))(x)(t_{k+1}-t_k)$ for any $x\in\Omega$ and $n\in\N$.
Since the function $Tf$ is continuous in $[a,b]\times\Omega$, the same arguments as above show that
$(Tg_n)(x)$ converges to $\int_a^b(Tf(s,\cdot))(x)ds$ as $n\to +\infty$, and we are done.
\end{proof}

\end{document}